%% file: ss.tex
\documentclass[10pt,letterpaper,twoside]{article}
\pdfoutput=1

\usepackage[letterpaper,left=1in,right=1in,top=1.5in,bottom=1.5in]{geometry}

\title{Spectral sequences, d\'ecalage, and the Beilinson $t$-structure}\author{Benjamin Antieau}\date{\today}\input{preamble}

\begin{document}

\maketitle

\begin{abstract}
    \noindent
    This paper explains the theory of spectral sequences via d\'ecalage and the Beilinson $t$-structure.
\end{abstract}

\vspace{60pt}

\begin{figure}[h]
    \centering
    \makebox[\textwidth]{
        \begin{tabular}{||l l l l l||}
            \hline
            $\pi_*\F^\star
            M$&$\E^1_{s,t}=\pi_{s+t}\gr^{-s}M\Rightarrow\pi_{s+t}M$&$(-r,r-1)$&$\F^s{\pi_nM}=\im(\pi_n\F^sM\rightarrow\pi_nM)$&$\gr^s\pi_nM\iso\E^\infty_{-s,s+n}$\\
            \hline
            $\pi_*\F_\star
            M$&$\E^1_{s,t}=\pi_{s+t}\gr_sM\Rightarrow\pi_{s+t}M$&$(-r,r-1)$&$\F_s\pi_nM=\im(\pi_n\F_sM\rightarrow\pi_nM)$&$\gr_s\pi_nM\iso\E^\infty_{s,-s+n}$\\
            \hline
            $\H^*\F^\star
            M$&$\E_1^{s,t}=\H^{s+t}\gr^sM\Rightarrow\H^{s+t}M$&$(r,-r+1)$&$\F^s\H^nM=\im(\H^n\F^sM\rightarrow\H^nM)$&$\gr^s\H^nM\iso\E_\infty^{s,-s+n}$\\
            \hline
            $\H^*\F_\star
            M$&$\E_1^{s,t}=\H^{s+t}\gr_{-s}M\Rightarrow\H^{s+t}M$&$(r,-r+1)$&$\F_s\H^nM=\im(\H^n\F_sM\rightarrow\H^nM)$&$\gr_s\H^nM\iso\E_\infty^{-s,s+n}$\\
            \hline
        \end{tabular}
    }
    \caption{Spectral sequence conventions for decreasing filtrations $\F^\star$ and increasing
    filtrations $\F_\star$. We use $\Rightarrow$ even when no convergence is implied.}
    \label{fig:cover_ssconventions}
\end{figure}

        \begin{figure}[h]
    \centering
    \makebox[\textwidth]{
        \begin{tabular}{||l l l||}
            \hline
            $\pi_*\F^\star
            M$&$'\E^2_{s,t}=\E^1_{-t,s+2t}=\pi_{s+t}\gr^{t}M\Rightarrow\pi_{s+t}M$&$\gr^s\pi_nM\iso\E^\infty_{-s+n,s}$\\
            \hline
            $\pi_*\F_\star
            M$&$'\E^2_{s,t}=\E^1_{-t,s+2t}=\pi_{s+t}\gr_{-t}M\Rightarrow\pi_{s+t}M$&$\gr_s\pi_nM\iso\E^\infty_{s+n,-s}$\\
            \hline
            $\H^*\F^\star
            M$&$'\E_2^{s,t}=\E_1^{-t,s+2t}=\H^{s+t}\gr^{-t}M\Rightarrow\H^{s+t}M$&$\gr^s\H^nM\iso\E_\infty^{s+n,-s}$\\
            \hline
            $\H^*\F_\star
            M$&$'\E_2^{s,t}=\E_1^{-t,s+2t}=\H^{s+t}\gr_{t}M\Rightarrow\H^{s+t}M$&$\gr_s\H^nM\iso\E_\infty^{-s+n,s}$\\
            \hline
        \end{tabular}
    }
    \caption{Reindexing to $'\E^2$. The differentials have the same bidegree as
    in the $\E^1$-spectral sequences from Figure~\ref{fig:cover_ssconventions}. The
    filtrations on the abutment are also defined in the same way as in the $\E^1$ spectral
    sequences.}
    \label{fig:cover_E2ssconventions}
\end{figure}

\begin{figure}[h]
    \centering
    \makebox[\textwidth]{
        \begin{tabular}{||l l l l l||}
            \hline
            $\pi_*\F^\star
            M$&$\E^1_{s,t}=\pi_{s}\gr^{t}M\Rightarrow\pi_{s}M$&$(-1,r)$&$\F^s{\pi_nM}=\im(\pi_n\F^sM\rightarrow\pi_nM)$&$\gr^s\pi_nM\iso\E^\infty_{n,s}$\\
            \hline
            $\pi_*\F_\star
            M$&$\E^1_{s,t}=\pi_{s}\gr_{-t}M\Rightarrow\pi_{s}M$&$(-1,r)$&$\F_s\pi_nM=\im(\pi_n\F_sM\rightarrow\pi_nM)$&$\gr_s\pi_nM\iso\E^\infty_{n,-s}$\\
            \hline
            $\H^*\F^\star
            M$&$\E_1^{s,t}=\H^{s}\gr^{-t}M\Rightarrow\H^{s}M$&$(1,-r)$&$\F^s\H^nM=\im(\H^n\F^sM\rightarrow\H^nM)$&$\gr^s\H^nM\iso\E_\infty^{n,-s}$\\
            \hline
            $\H^*\F_\star
            M$&$\E_1^{s,t}=\H^{s}\gr_{t}M\Rightarrow\H^{s}M$&$(1,-r)$&$\F_s\H^nM=\im(\H^n\F_sM\rightarrow\H^nM)$&$\gr_s\H^nM\iso\E_\infty^{n,s}$\\
            \hline
        \end{tabular}
    }
    \caption{Adams indexing spectral sequence conventions.}
    \label{fig:cover_adamsssconventions}
\end{figure}

\newpage

\section{Introduction}

There has been a recent cultural movement best portrayed as:
$$\text{``there is a spectral sequence''}\Rightarrow\text{``there is a
filtration''}.$$
This move is due to the fact that filtrations admit better functorial
properties than spectral sequences.
For example, Bhatt, Morrow, and Scholze use quasisyntomic descent
in~\cite{bms2} to define new
filtrations on topological Hochschild homology and related invariants of
$p$-adic rings by reducing to particularly simple cases.
This requires descending a filtration, where descending a spectral sequence
would not classically be meaningful.

The Beilinson $t$-structure plays an important role in~\cite{bms2} and
one can also use the Beilinson $t$-structure to construct
related filtrations on integral periodic cyclic homology, as
in~\cite{antieau-derham}, or on $\TP$ of schemes in characteristic $p$, as
in~\cite{an1}.
Given a stable $\infty$-category $\Cscr$ which admits sequential limits and equipped with a
$t$-structure,
the Beilinson $t$-structure is a $t$-structure on $\F\Cscr$, the $\infty$-category of filtered
objects in $\Cscr$, whose heart is the abelian category
$\Ch^\bullet(\Cscr^\heart)$ of cochain complexes.
Given a filtration $\F^\star\in\F\Cscr$,
the cochain complexes appearing as the homotopy objects with respect to the
Beilinson $t$-structure are precisely the cochain complexes occurring on the
$\E^1$-page of the associated spectral sequence.
This $t$-structure was first defined and studied by Beilinson
in~\cite{beilinson}; see~\cite{bms2,ariotta,raksit} for recent treatments.

The goal of this paper is to explain the theory of spectral sequences as seen
through the lens of the Beilinson $t$-structure. The main construction is the d\'ecalage
$\Dec(\F)^\star$ of a filtration $\F^\star$. If $\Cscr$ additionally
admits sequential colimits, then every
filtered object $\F^\star$ of $\Cscr$ admits an underlying object $|\F^\star|$ obtained by forgetting the filtration, or taking the colimit.
Each connective cover $\tau_{\geq n}^\B(\F^\star)$ in
the Beilinson $t$-structure is, by definition, itself a filtered object in $\Cscr$. Thus, given a filtered object $\F^\star$
one can obtain a new filtered object $\Dec(\F)^\star$ by taking the
underlying objects of the
Whitehead tower $$\cdots\rightarrow|\tau_{\geq n+1}^\B(\F^\star)|\rightarrow|\tau_{\geq n}^\B(\F^\star)|\rightarrow|\tau_{\geq
n-1}^\B(\F^\star)|\rightarrow\cdots.$$ This construction defines an endomorphism
$\Dec\colon\F\Cscr\rightarrow\F\Cscr$, the d\'ecalage functor.
The main result of the paper is that d\'ecalage turns the page. Specifically, if $\Cscr$ is a
stable $\infty$-category equipped with a $t$-structure and $\F^\star\in\F\Cscr$, there is an
associated spectral sequence defined, for example, in~\cite[Sec.~1.2.2]{ha}.
Below, an isomorphism of pages between two spectral sequences includes compatibility of the
isomorphism the differentials.

\begin{theorem}[See Theorem~\ref{thm:decalage}]\label{thm:main}
    Let $\Cscr$ be a stable $\infty$-category which admits sequential limits and colimits and fix a
    $t$-structure on $\Cscr$. For $r\geq 1$,
    the $\E^r$-page of the spectral sequence for $\Dec(\F^\star)$ is naturally
    isomorphic to the $\E^{r+1}$-page of the spectral sequence for $\F^\star$.
\end{theorem}

Let $\Dec^{(r-1)}(\F^\star)$ denote the $(r-1)$-fold composition of the
d\'ecalage construction.

\begin{corollary}
    The $\E^r$-page of the spectral sequence for $\F^\star$ is naturally
    isomorphic to the $\E^1$-page of the spectral sequence for
    $\Dec^{(r-1)}(\F^\star)$.
\end{corollary}

The corollary leads to a quick definition of the spectral sequence associated to a filtered object
$\F^\star\in\F\Cscr$.

\begin{definition}\label{def:main}
    Let $\F^\star$ be a filtered object in a stable $\infty$-category $\Cscr$
    with a $t$-structure and suppose that $\Cscr$ admits sequential limits and colimits.
    Write the associated coherent cochain complex $$\cdots\rightarrow\gr^{-1}[-1]\rightarrow\gr^0
    \rightarrow\gr^1 [1]\rightarrow\cdots$$
    obtained by pasting together the boundary maps $$\F^n
    /\F^{n+2}\rightarrow\gr^n\rightarrow\gr^{n+1}[1]$$
    and take homotopy objects with respect to the $t$-structure on $\Cscr$ to
    obtain cochain complexes
    $$\cdots\rightarrow\pi_{t+1}\gr^{-1} M\rightarrow\pi_t\gr^0
    M\rightarrow\pi_{t-1}\gr^1 M\rightarrow\cdots.$$
    These form the $\E^1$-page of the spectral sequence associated to $\F^\star$, where
    $\E^1_{s,t}=\pi_{s+t}\gr^{-s}$.
    To define the $\E^{r+1}$-page, for some $r\geq 0$, apply the d\'ecalage construction above $(r-1)$
    times and take the $\E^1$-page of the resulting filtration,
    applying the transformation $$\begin{pmatrix}-r+1&-r\\r&r+1\end{pmatrix}$$ to the indices.
\end{definition}

Definition~\ref{def:main} gives a sequence of pages for which it is not hard to prove
that the cohomology groups of one page agree with the terms on the next page. (See
Remark~\ref{rem:graded_decalage}.) In other words,
Definition~\ref{def:main} produces a spectral sequence in the sense of
any standard textbook, such as~\cite{mccleary}.
The homotopy coherent nature of the d\'ecalage functor makes it possible to make
$\infty$-categorical arguments about the nature of the spectral sequence attached to a filtered
object via Definition~\ref{def:main}.
For example, the multiplicative properties of these spectral sequences follow immediately from the
fact that the d\'ecalage functor admits a lax symmetric monoidal structure. See
Hedenlund's thesis~\cite{hedenlund-thesis} and also
Section~\ref{sec:multiplicativity}.

The content of Theorem~\ref{thm:main} is that Definition~\ref{def:main} yields
$\E^{r+1}$-pages and $d^{r+1}$-differentials which agree naturally for $r\geq 0$ with the usual
definition of the spectral sequence of a filtration, for which our reference is~\cite[Const.~1.2.2.6]{ha}.

The Beilinson $t$-structure
implements in a homotopy coherent way Deligne's d\'ecalage functor, which was
used to construct canonical mixed Hodge structures on the singular cohomology of
complex algebraic varieties in~\cite{deligne-hodge-2,deligne-hodge-3}.
Suppose that $\F^\star M^\cs$ is a strictly filtered cochain complex and
let $\Dec(\F)^\star M^\cs$ be Deligne's d\'ecalage construction
(see Definition~\ref{def:delignedecalage}), which is a new strict filtration on $M^\bullet$. One can put on each $\Dec(\F)^s M^\cs$ a secondary
filtration $\G^\diamond\Dec(\F)^s M$ in such a way that
$\G^\diamond\Dec(\F)^\star M$ becomes a strictly bifiltered cochain
complex.

\begin{theorem}[See
    Theorem~\ref{thm:decalagecomparison}]\label{thm:main_comparison}
    There is a natural map $\G^\diamond\Dec(\F)^\star M^\cs\rightarrow\F^\star
    M$ inducing an equivalence $$\G^\diamond\Dec(\F)^\star
    M^\cs\we\tau_{\geq\star}^\B(\F)M.$$
    In particular, Delgine's strict d\'ecalage functor is equivalent to the
    $\infty$-categorical d\'ecalage functor.
\end{theorem}

An $\infty$-categorical form of d\'ecalage was constructed by Levine in~\cite{levine_anss} in
the context of cosimplicial spectra. Levine does not use the language of the Beilinson
$t$-structure, but does prove analogues of
Theorems~\ref{thm:main} and~\ref{thm:main_comparison};
see~\cite[Prop.~6.3]{levine_anss} and~\cite[Rem.~6.4]{levine_anss}.
Theorem~\ref{thm:main_comparison}
was discovered independently by the author around 2019 and by Bhatt, Morrow, and
Scholze~\cite{bms2}. The author discovered Theorem~\ref{thm:main} as stated shortly thereafter
and communicated it to several people. See for
example~\cite{ariotta,gikr,hedenlund-thesis,mor-profinite}, which contain expositions of some of
the ideas of this paper. In particular, the paper~\cite{gikr} of Gheorghe--Isaksen--Krause--Ricka
shows that if $\Tot^\star(\MU^\bullet)$ denotes the Adams tower associated to the \v{C}ech complex
of $\bS\rightarrow\MU$, then $\Dec(\Tot^\star(\MU^\bullet))$ is what is now called the even
filtration $\bS_\ev$ on the sphere spectrum in the language of~\cite{hrw}.
Recently, Tyler Lawson~\cite{lawson-filtered} has given a generalized version of a d\'ecalage construction
which applies more generally to filtrations in a broad class of $\infty$-categories, not just
stable ones. Lawson gives a sketch of the agreement of $\E^r(\Dec(\F))$ and $\E^{r+1}(\F)$
in~\cite[Prop.~9.17]{lawson-filtered}.

Besides the theoretical results mentioned above, this paper
is intended as a brief manual on spectral sequences, so more is included than strictly necessary
for the proofs of the main theorems. For example, we include discussions of indexing conventions and convergence,
which do not in fact require the methods developed here. The convergence discussion is especially
simplistic as we are awaiting forthcoming work of Hedenlund, Krause, and Nikolaus on a new approach
to conditional convergence.

Two applications of the d\'ecalage approach to
spectral sequences are given, namely to multiplicativity of spectral sequences and to the proof that the
two standard constructions of the Atiyah--Hirzebruch spectral sequence agree from the $\E_2$-page
on.

\paragraph{Additional related work.}
Our perspective places filtered objects at the heart of the study of spectral sequences, including
the higher pages. By contrast, refer to work of Livernet--Whitehouse~\cite{livernet-whitehouse} or
Cirici--Egas-Santander--Livernet~\cite{cirici-egas-santander-livernet} who
study homotopy theories of extended spectral sequences and define d\'ecalage functors for these
theories. In another direction, Cirici--Guill\'en~\cite{cirici-guillen} use Deligne's d\'ecalage to denote certain
localizations of the homotopy theory of filtered complexes with applications to mixed Hodge theory.

\paragraph{Acknowledgments.}
Peter Scholze suggested to me the problem of constructing
a ``BMS'' filtration on periodic cyclic homology and related invariants
without $p$-completeness assumptions. This led
to~\cite{antieau-derham} and to an appreciation of the Beilinson $t$-structure.
Akhil Mathew asked how the d\'ecalage used there relates to that defined by
Deligne in~\cite{deligne-hodge-2}, leading to this paper. I thank both for their motivating questions.
Additionally, I thank Achim Krause and Thomas Nikolaus for explaining parts of their work on
convergence.

This paper would never have been completed without the interest of Julie Creemers, Lennart Meier,
and Itamar Mor. Additionally, Julie Creemers and Alice Hedenlund provided very helpful comments on
a draft, for which I am very grateful.

I was supported by NSF grants DMS-2120005,
DMS-2102010, and DMS-2152235, by Simons Fellowships 666565 and 00005925, and by the Simons
Collaboration on Perfection. This paper was completed while I was a guest at the MPIM.

\section{$t$-structures}\label{sec:t}

Stable $\infty$-categories are the natural higher categorical models of triangulated categories.
We include a brief overview of the theory. For proofs and more details, see~\cite[Chap.~1]{ha}
or~\cite[App.~C]{sag}.

\begin{definition}
    \begin{enumerate}
        \item[(a)] An $\infty$-category is
            pointed if it admits an
            object $0$ which is both initial and terminal.
        \item[(b)] If $\Cscr$ is pointed and
            admits finite colimits, then there is a suspension functor
            $\Sigma\colon\Cscr\rightarrow\Cscr$ obtained by taking the pushout
            of the functors $0\leftarrow\id\rightarrow 0$. The
            Spanier--Whitehead $\infty$-category of $\Cscr$ is
            $\SW(\Cscr)\we\colim(\Cscr\xrightarrow{\Sigma}\Cscr\xrightarrow{\Sigma}\cdots)$.
            An $\infty$-category is
            prestable if it is
            pointed, admits finite colimits, the suspension functor
            $\Sigma$ is fully faithful, and the essential image of the fully
            faithful functor
            $\Cscr\rightarrow\SW(\Cscr)$ is closed under extensions.
        \item[(c)] An $\infty$-category $\Cscr$ is stable if it is pointed,
            admits finite colimits, and the suspension functor
            $\Sigma$ is an equivalence. If $\Cscr$ is prestable, then
            $\SW(\Cscr)$ is prestable.
    \end{enumerate}
\end{definition}

\begin{definition}
    A functor $\Cscr\rightarrow\Dscr$ between stable $\infty$-categories is
    exact if it preserves finite colimits.
\end{definition}

\begin{fact}
    If $\Cscr$ is a stable $\infty$-category, then its homotopy category
    $\Ho(\Cscr)$ admits a natural triangulated structure. If
    $f\colon\Cscr\rightarrow\Dscr$ is an exact functor between stable
    $\infty$-categories, then $\Ho(f)\colon\Ho(\Cscr)\rightarrow\Ho(\Dscr)$ is
    a triangulated functor.
\end{fact}

No non-trivial pointed $1$-category is stable or prestable. Indeed, in a pointed
$1$-category the suspension functor is the constant functor on $0$. In
particular, triangulated categories are neither prestable nor stable.

\begin{example}
    \begin{enumerate}
        \item[(a)] The $\infty$-category $\Sp$ is the stabilization (in $\Pr^\L$) of the
            $\infty$-category $\Sscr$ of anima. This is a symmetric monoidal stable
            $\infty$-category which is compactly generated by its unit, the sphere spectrum $\bS$.
            Thus, we will also write $\D(\bS)$ for $\Sp$ as $\Mod_{\bS}(\Sp)\we\Sp$.
        \item[(b)] If $A$ is an $\bA_\infty$-ring spectrum, then $\LMod_A(\D(\bS))=\D(A)$ is the
            $\infty$-category of left $A$-module spectra. If $A$ is a dg algebra, then the homotopy
            category of $\D(A)$ is equivalent to the usual triangulated category of left dg modules over
            $A$ up to quasi-isomorphism. In particular, this holds when $A$ is a discrete
            commutative ring.
        \item[(c)] If $X$ is a quasicompact and quasiseparated scheme, then $\D_\qc(X)=\lim_{\Spec
            R\rightarrow X}\D(R)$ defines a stable $\infty$-category whose homotopy category is the
            familiar triangulated category of complexes of sheaves of $\Oscr_X$-modules with
            quasicoherent cohomology.
    \end{enumerate}
\end{example}

The theory of $t$-structures plays a fundamental role in this paper.

\begin{definition}
    A $t$-structure on a stable $\infty$-category $\Cscr$ consists of a pair
    $(\Cscr_{\geq 0},\Cscr_{\leq 0})$ of full subcategories of $\Cscr$
    satisfying the following conditions:
    \begin{enumerate}
        \item[(a)] $\Cscr_{\geq 0}[1]\subseteq\Cscr_{\geq 0}$ and $\Cscr_{\leq
            0}\subseteq\Cscr_{\leq 0}[1]$;
        \item[(b)] if $X\in\Cscr_{\geq 0}$ and $Y\in\Cscr_{\leq 0}$, then
            the mapping space $\Map_\Cscr(X,Y[-1])$ is contractible;
        \item[(c)] every object $X$ in $\Cscr$ fits into a fiber sequence
            $W\rightarrow X\rightarrow Z$ where $W\in\Cscr_{\geq 0}$ and
            $Z[1]\in\Cscr_{\leq 0}$.
    \end{enumerate}
    It is not hard to see that to give a $t$-structure on $\Cscr$ is equivalent to give one on its
    triangulated homotopy category $\Ho(\Cscr)$.
\end{definition}

\begin{remark}
    If $(\Cscr_{\geq 0},\Cscr_{\leq 0})$ is a $t$-structure, then the inclusions of
    $\Cscr_{\geq 0}$ and $\Cscr_{\leq 0}$ into $\Cscr$ admit right and left
    adjoints, respectively. This follows from a Bousfield localization argument using
    the orthogonality relation in (b) to show that given $X\in\Cscr$ the sequence
    $W\rightarrow X\rightarrow Z$ is unique up to homotopy. The functor
    $\Cscr\rightarrow\Cscr_{\geq 0}$ is $\tau_{\geq 0}$, while
    $\Cscr\rightarrow\Cscr_{\leq 0}$ is $\tau_{\leq 0}$. No care will be taken to
    distinguish between these functors and the corresponding colocalization and
    localization functors $\Cscr\rightarrow\Cscr_{\geq 0}\hookrightarrow\Cscr$ and
    $\Cscr\rightarrow\Cscr_{\leq 0}\hookrightarrow\Cscr$.

    More generally, one defines subcategories $\Cscr_{\geq n}\we\Cscr_{\geq 0}[n]$
    and $\Cscr_{\leq n}\we\Cscr_{\leq 0}[n]$
    for any integer $n$. There are corresponding truncation functors $\tau_{\geq
    n}\colon\Cscr\rightarrow\Cscr_{\geq n}$ and $\tau_{\leq
    n}\colon\Cscr\rightarrow\Cscr_{\leq n}$ for any $n$.
\end{remark}

\begin{definition}
    The objects of $\Cscr_{\geq 0}$ are called connective while those of
    $\Cscr_{\leq 0}$ are coconnective or $0$-truncated. More generally, objects of
    $\Cscr_{\geq n}$ are $n$-connective and those of $\Cscr_{\leq n}$ are
    called $n$-truncated, or at times $n$-coconnective.
    Note: these notions depend on the chosen $t$-structure on $\Cscr$.
    Often, as is the case in spectra or $\D(\bZ)$, this $t$-structure is not
    specified, in which case the standard $t$-structure is always implicit.
\end{definition}

The following lemma guarantees that the order of application of truncation and connective cover
functors is irrelevant.

\begin{lemma}
    For any $m\leq n$, define $\Cscr_{[m,n]}=\Cscr_{\geq m}\cap\Cscr_{\leq n}$.
    There is a natural equivalence $\tau_{\geq m}\circ\tau_{\leq
    n}\we\tau_{\leq n}\circ\tau_{\geq m}$ of functors
    $\Cscr\rightarrow\Cscr_{[m,n]}$.
\end{lemma}

\begin{remark}
    Additionally, there are natural equivalences $\Sigma\circ\tau_{\geq
    n}\we\tau_{\geq n+1}\circ\Sigma$ and $\Sigma\circ\tau_{\leq n}\we\tau_{\leq
    n+1}\circ\Sigma$ and similarly for higher powers of $\Sigma$.
\end{remark}

\begin{definition}
    Given a $t$-structure on $\Cscr$, the heart of the $t$-structure is
    the intersection $\Cscr^\heart=\Cscr_{\geq 0}\cap\Cscr_{\leq 0}$. This turns
    out to be an abelian category, although it might be trivial. Given any integer
    $n$, one takes $\pi_n\colon\Cscr\rightarrow\Cscr^\heart$ to be the functor
    $\Sigma^{-n}\circ\tau_{\geq n}\circ\tau_{\leq n}\we\Sigma^{-n}\circ\tau_{\leq
    n}\circ\tau_{\geq n}\we\tau_{\leq 0}\circ\tau_{\geq 0}\circ\Sigma^{-n}\we\tau_{\geq
    0}\circ\tau_{\leq 0}\circ\Sigma^{-n}$.
\end{definition}

\begin{remark}
    Given a fiber sequence $X\rightarrow Y\rightarrow Z$ in $\Cscr$, there is an
    induced long exact sequence
    $$\cdots\rightarrow\pi_nX\rightarrow\pi_nY\rightarrow\pi_nZ\rightarrow\pi_{n-1}X\rightarrow\cdots$$
    in $\Cscr^\heart$.
    These sequences give a hands-on way for understanding stable
    $\infty$-categories using more familiar abelian categories.
\end{remark}

\begin{remark}
    If $(\Cscr_{\geq 0},\Cscr_{\leq 0})$ is a $t$-structure on a stable
    $\infty$-category $\Cscr$, then $\Cscr_{\geq 0}$ is a prestable
    $\infty$-category which is closed under extensions in $\Cscr$. The converse is
    true under the presence of enough colimits using a Bousfield localization argument.
    See~\cite[Prop.~1.4.4.11]{ha}.
\end{remark}

\begin{example}
    \begin{enumerate}
        \item[(a)] The $\infty$-category $\Sp\we\D(\bS)$ of spectra admits a $t$-structure 
            where $\Sp_{\geq 0}$ is the full subcategory of spectra $X$ such that
            $\pi_nX=0$ for $n<0$ while $\Sp_{\leq 0}$ is the full subcategory of $X$ where
            $\pi_nX=0$ for $n>0$. The heart $\Sp^\heart$ is the abelian category
            $\Mod_{\bZ}$ of abelian groups.
        \item[(b)] More generally, if $A$ is any connective $\bE_1$-ring spectrum, then $\D(A)$
            inherits a $t$-structure where $\D(A)_{\geq 0}$ consists of the $A$-module
            spectra $X$ where $\pi_nX=0$ for $n<0$ and $\D(A)_{\leq 0}$ consists of those
            where $\pi_nX=0$ for $n>0$. The connectivity assumption on $A$ guarantees that
            if $X$ is an $A$-module spectrum, then $\tau_{\geq 0}X$, where the truncation
            is taken in spectra, inherits a canonical $A$-module spectrum structure.
            The heart of this $t$-structure is naturally equivalent to the abelian category
            $\Mod_{\pi_0A}$ of left $\pi_0A$-modules.
            If $A$ is discrete, then this $t$-structure on $\D(A)$ recovers the usual
            $t$-structure on the derived $\infty$-category of $A$, which encodes the ``good
            truncations'' of complexes.
        \item[(c)] If $\Cscr$ is a stable $\infty$-category equipped with a $t$-structure
            $(\Cscr_{\geq 0},\Cscr_{\leq 0})$ and if $I$ is an $\infty$-category, then the pair
            $(\Fun(I,\Cscr_{\geq 0}),\Fun(I,\Cscr_{\leq 0}))$ forms a $t$-structure on
            $\Fun(I,\Cscr)$ with heart
            $\Fun(I,\Cscr)^\heart\we\Fun(I,\Cscr^\heart)\we\Fun(\Ho(I),\Cscr^\heart)$, where
            $\Ho(I)$ denotes the homotopy category of $I$.
    \end{enumerate}
\end{example}

\section{Beilinson's $t$-structure}\label{sec:beilinson}

Filtered objects in a stable $\infty$-category have been studied in several places.
See~\cite{ariotta,gwilliam-pavlov,lurie-rotation,ha,moulinos-filtrations,raksit}.

Let $\bZ$ denote the ordered set of integers viewed as a category so that there
is a (unique) morphism $m\rightarrow n$ if and only if $m\leq n$, in which case the
morphism is unique. Let $\bZ^\op$ be the opposite of $\bZ$.

\begin{definition}\label{def:filtration}
    Let $\Cscr$ be an $\infty$-category. The $\infty$-category of decreasing
    filtrations in $\Cscr$ is $$\F\Cscr=\Fun(\bZ^\op,\Cscr).$$ The $\infty$-category of
    increasing filtrations in $\Cscr$ is $\Fun(\bZ,\Cscr)$.
\end{definition}

\begin{remark}
    There is an equivalence $\bZ\we\bZ^\op$ obtained by sending $m$ to $-m$. The
    theory of increasing filtrations and decreasing filtrations are thus
    equivalent, so we focus below on decreasing filtrations.
\end{remark}

A decreasing filtration in $\Cscr$ is represented as an infinite sequence
$$\F^\star \colon\cdots\rightarrow\F^{s+1}\rightarrow\F^s\rightarrow\F^{s-1}\rightarrow\cdots.$$
Included in the data of a functor $\bZ^\op\rightarrow\Cscr$ is homotopy coherence data expressing
for example a specific homotopy between $\F^{s+1}\rightarrow\F^{s-1}$ and the
composition of $\F^{s+1}\rightarrow\F^s$ with $\F^s\rightarrow\F^{s-1}$.

\begin{example}
    If $\Cscr\we\D(\bZ)$ is the $\infty$-category of $\bZ$-module spectra, then
    $\FD(\bZ)$ is the $\infty$-category of filtered $\bZ$-module spectra. The
    homotopy category of $\FD(\bZ)$ is what is often called the derived
    category of filtered complexes. In fact, if $\F\Mod_{\bZ}$ is the
    Grothendieck abelian
    category of filtered abelian groups, then $\FD(\bZ)\we\D(\F\Mod_{\bZ})$ is
    its derived $\infty$-category. This follows for example from
    Proposition~\ref{prop:inheritance} below.
\end{example}

\begin{warning}
    By Definition~\ref{def:filtration}, a filtered object of
    $\F\Mod_{\bZ}$ is an arbitrary $\bZ^\op$-indexed sequence of maps of abelian groups. There is
    no requirement that the transition maps in the sequence be injective. When
    they are, the filtration is called strict. The category of strict
    filtrations on $\Mod_{\bZ}$ has kernels and cokernels, but is not abelian.
\end{warning}

Often, one considers filtrations of specific objects.

\begin{definition}
    A decreasing filtration on an object $X$ of $\Cscr$ is a
    filtration $\F^\star$ together with a map $\F^\star\rightarrow X$ of filtrations, where
    $X$ denotes the constant decreasing filtration on $X$.
    A filtration $\F^\star$ on $X$ is exhaustive if $X$ is a colimit of
    the diagram $\F^\star$. In analogy with the ordinal $\omega+1$, one can consider the
    $\infty$-category of pairs $(\F^\star\rightarrow X)$ of filtrations on objects of $\Cscr$ as the functor category
    $\Fun(\bZ^\op+1,\Cscr)$, where $\bZ^\op+1$ is the union of $\bZ^\op$ and a new, terminal element
    $-\infty$. Of course, we have
    $\bZ^\op+1\we(\bZ^\op)^{\vartriangleright}\we(\bZ^\vartriangleleft)^\op$.
\end{definition}

The picture to have in mind here is a large commutative diagram
$$\xymatrix{
    \cdots\ar[r]&\F^{s+1}\ar[dr]\ar[r]&\F^s\ar[d]\ar[r]&\F^{s+1}\ar[dl]\ar[r]&\cdots\\
    &&X.&&
}$$

\begin{example}
    If $\Cscr$ admits sequential colimits, then any filtration $\F^\star$ can
    be viewed as giving a filtration on $\F^{-\infty}=\colim_s\F^s$. This
    construction will appear frequently, so we denote it by
    $|\F^\star|=\F^{-\infty}$. If $\F^\star$ is a filtration on $X$, then there
    is a canonical map $|\F^\star|\rightarrow X$, which is an equivalence if
    the filtration is exhaustive. If $\Cscr$ has sequential colimits, the functor
    $|-|\colon\F\Cscr\rightarrow\Cscr$ is left adjoint to the constant filtration
    functor $\Cscr\rightarrow\F\Cscr$.
\end{example}

\begin{notation}
    We use the notation $\F^\star$ for a generic filtration to emphasize the primacy of the
    filtration as opposed to a possible object being filtered. Sometimes, we will write $\F^\star
    X$ for a generic filtration as well, in which case, unless specified otherwise, $X=|\F^\star|$.
\end{notation}

\begin{definition}
    If $\Cscr$ has a final object $\ast$ and admits cofibers and if $\F^\star$ is a decreasing
    filtration in $\Cscr$, the associated graded pieces are
    $$\gr^s_\F=\mathrm{cofiber}(\F^{s+1}\rightarrow\F^s)=\frac{\F^{s+1}}{\F^s}$$ for $s\in\bZ$.
    The collection of associated graded pieces forms a graded object
    $\gr^\star_\F$ of $\Cscr$, i.e., an object of $\Fun(\bZ^\delta,\Cscr)$, where
    $\bZ^\delta$ denotes the set $\bZ$ viewed as a $1$-category with only
    identity morphisms.
    If $\F^\star$ is a filtration then its weights are the integers $s$
    for which $\gr^s_\F\we\F^s/\F^{s+1}$ is non-zero.
\end{definition}

The underlying goal of spectral sequences, and this paper, is to explain how to understand $X$ from
information about the graded pieces $\gr^\star_\F X$ of some filtration $\F^\star
X$ on $X$. This understanding is what spectral sequences `do'. For the
remainder of this paper, filtrations are studied in the context of {\em stable}
$\infty$-categories.

\begin{definition}
    A filtration $\F^\star$ in $\Cscr$ is complete if $\F^\infty=\lim_s\F^s\we 0$, the initial object of
    $\Cscr$. In particular, the limit exists.
    Let $\widehat{\F}\Cscr\subseteq\F\Cscr$ be the full subcategory of complete
    filtrations. If $\Cscr$ admits sequential limits, then the inclusion
    $\widehat{\F}\Cscr\subseteq\F\Cscr$ admits a left adjoint, called completion, given by
    $\widehat{\F}^\star=\F^\star/\F^\infty$.
    Let $\F^\star$ be a decreasing filtration on an object $X$ of $\Cscr$.
    The completion of $X$ with respect to the filtration is $\widehat{X}=\lim_s(X/\F^sX)\we
    X/(\lim_s\F^sX)\we X/(\F^{\infty}X)$. The completion $\widehat{\F}^\star$ is a complete
    filtration on $\widehat{X}$.
\end{definition}

\begin{example}
    Let $p^\star\bZ$ denote the $p$-adic filtration of $\bZ$. Specifically,
    $p^s\bZ$ is the ideal $(p^s)\subseteq\bZ$ for $s\geq 0$ and $p^s\bZ=\bZ$
    for $s\leq 0$. The graded pieces are $\gr^s\bZ\we (p^s)/(p^{s+1})\iso\bZ/p$
    for $s\geq 0$ and $0$ for $s<0$. This
    filtration is exhaustive, but not complete:
    $\F^\infty=\lim_s(p^s)\we(\bZ_p/\bZ)[-1]$. The completion is the $p$-adic
    filtration $p^\star\bZ_p$ of $\bZ_p$.
\end{example}

\begin{remark}
    There is a duality between complete exhaustive filtrations $\F^\star X$ on $X$ and
    non-complete filtrations on $0$ with limit $X$. Indeed, if $\F^\star X$ is
    complete and exhaustive, consider the cone of $\F^\star X\rightarrow
    X$, which is the filtration $\frac{X}{\F^\star X}$. The limit of this
    filtration is
    $\lim_s X/F^sX\we X$ and the colimit is $\colim_s X/F^sX\we 0$.
    Conversely, if $\F^\star$ is an exhaustive filtration on $0$ with with
    limit $Z$, then taking the fiber of $Z\rightarrow\F^\star$ gives a complete
    exhaustive filtration on $Z$. As example, consider the two perspectives on
    the completeness of the $p$-adic filtration: $\lim_s p^s\bZ_p=0$ versus
    $\bZ_p\iso\lim_s \bZ_p/p^s$.
\end{remark}

\begin{example}\label{ex:stupid}
    Let $M^\bullet\in\Ch^\bullet(\bZ)$ be a cochain complex. The stupid
    filtration
    $\sigma^\star M^\bullet$ on $M$ is the filtered cochain complex with $\sigma^s
    M^\bullet=M^{\bullet\geq s}$. The graded pieces are $\gr^sM^\bullet\iso
    M^s$ placed in cohomological degree $s$. The stupid filtration is a complete
    exhaustive filtration on $M^\bullet$ in $\Ch^\bullet(\bZ)$. Looking at
    underlying homotopy types produces a complete exhaustive filtration $\sigma^\star M$ on the
    image $M$ of $M^\bullet$ in the derived $\infty$-category $\D(\bZ)$.
    The graded pieces are $\gr^sM\we M^s[-s]$. We return to this example below in
    Construction~\ref{const:realization}.
\end{example}

\begin{example}\label{ex:whitehead}
    Let $X\in\D(\bS)$ be a spectrum and let $\F^\star X=\tau_{\geq\star}X$ be its
    Whitehead tower: $$\cdots\rightarrow\tau_{\geq s+1}X\rightarrow\tau_{\geq
    s}X\rightarrow\tau_{\geq s-1}X\rightarrow\cdots.$$ This is a complete and
    exhaustive filtration on $X$; the graded pieces are $\gr^sX\we\pi_sX[s]\in\D(\bS)$.
\end{example}

\begin{warning}
    Whitehead towers are neither complete nor exhaustive in a general stable
    $\infty$-category with a $t$-structure. In the case of spectra, the usual,
    Postnikov $t$-structure is compatible with filtered colimits,
    so that the homotopy groups of a filtered colimit are
    the filtered colimit of the homotopy groups, and is left-complete, so that
    Postnikov towers converge. These properties are shared by many other
    $t$-structures which occur in algebra, geometry, and topology, but not all.
    For an extensive discussion, see~\cite[Sec.~2]{clausen-mathew}.
\end{warning}

\begin{example}
    Let $X$ be a CW complex, which we view as a space with a certain increasing
    filtration $\F_\star X$ by cells. The filtration is complete and exhaustive
    by definition. Taking $\bZ$-cochains yields a filtered complex
    with
    $$\F^s_{\mathrm{CW}}\R\Gamma(X,\bZ)=\mathrm{fib}(\R\Gamma(X,\bZ)\rightarrow\R\Gamma(X_{s-1},\bZ)),$$
    where $\R\Gamma(-,\bZ)$ denotes the singular cohomology functor (which is
    equivalent to sheaf cohomology for CW complexes). The choice of indexing
    here implies that $\F^0_{\mathrm{CW}}\R\Gamma(X,\bZ)\we\R\Gamma(X,\bZ)$
    since the cohomology of the empty set vanishes. The CW filtration is complete and
    exhaustive on $\R\Gamma(X,\bZ)$ with associated graded pieces
    $\gr^s\R\Gamma(X,\bZ)\we(\R\Gamma(X_{s-1},\bZ)/\R\Gamma(X_{s},\bZ))[-1]\we
    \prod_{\text{$s$ cells}}\bZ[-s]$. This filtration is used below in
    Section~\ref{sec:mysterious}.
\end{example}

\begin{notation}
    Let $\F^\star$ be a filtered object in a stable $\infty$-category. For
    $j\geq i$ finite, let $\gr^{[i,j)}_\F=\cofib(\F^{j}\rightarrow\F^i)=\tfrac{\F^i}{\F^j}$. 
    For example, $\gr^{[i,i)}_\F\we 0$ and $\gr_\F^{[i,i+1)}\we\gr^i_\F$.
    Additionally, let
    $\gr^{[i,\infty)}_\F=\cofib(\F^\infty\rightarrow\F^i)=\lim_{j\rightarrow\infty}\gr^{[i,j)}_\F$
    and let $\gr^{(-\infty,j)}=\colim_{i\rightarrow-\infty}\gr^{[i,j)}_\F$.
    The graded pieces for other intervals are defined analogously.
    When unambiguous, $\gr^i$ is used for $\gr^i_\F$.
\end{notation}

\begin{construction}[Filtration on gradeds]\label{const:filtered_gradeds}
    It often convenient to view $\gr^{[i,j)}_\F$ as admitting a residual filtration
    $$\cdots\rightarrow 0\rightarrow\frac{\F^{j-1}}{\F^j}\rightarrow\frac{\F^{j-2}}{\F^j}\rightarrow\cdots\rightarrow\frac{\F^i}{\F^j}\xrightarrow{=}\frac{\F^i}{\F^j}\rightarrow\cdots,$$
    where the term $\tfrac{\F^{j-1}}{\F^j}$ is in filtration weight $j-1$.
    This is a complete exhaustive filtration on $\gr^{[i,j)}_\F$ with natural identifications
    $$\gr^a(\gr^{[i,j)}_\F)\we\begin{cases}
        \gr^a_\F&\text{if $i\leq a<j$, and}\\
        0&\text{otherwise.}
    \end{cases}$$
\end{construction}

Lurie observed in~\cite[Rem.~1.2.2.3]{ha} that if $\F^\star$ is a filtration, then
the graded objects form a kind of cochain complex. Specifically, each fiber
sequence $$\gr^{s+1}\rightarrow\frac{\F^s}{\F^{s+2}}\rightarrow\gr^s$$
gives rise to a `differential' $\delta\colon\gr^s\rightarrow\gr^{s+1}[1]$.

\begin{lemma}\label{lem:nullhomotopy}
    If $\F^\star$ is a filtration, then there is a canonical
    nullhomotopy $\delta\circ\delta\we 0$ of maps from $\gr^s$ to
    $\gr^{s+2}[2]$.
\end{lemma}

\begin{proof}
    The upper left $3\times 3$ part of the diagram below is commutative with
    exact rows and columns:
    $$\xymatrix{
        \gr^2\ar[r]\ar@{=}[d]&\gr^{[1,3)}\ar[r]\ar[d]&\gr^1\ar[d]\ar[r]^{\delta}&\gr^2[1]\ar@{=}[d]\\
        \gr^2\ar[r]\ar[d]&\gr^{[0,3)}\ar[r]\ar[d]&\gr^{[0,2)}\ar[r]\ar[d]&\gr^2[1]\ar[d]\\
        0\ar[r]&\gr^0\ar@{=}[r]&\gr^0\ar[r]\ar[d]^\delta&0\ar[d]\\
        &&\gr^1[1]\ar[r]^{\delta}&\gr^2[2].
    }$$
    The rest of the diagram is obtained by taking horizontal and vertical
    cofibers. The bottom right square expresses the nullhomotopy $\delta\circ\delta\we 0$.
\end{proof}

It follows that the sequence of morphisms
\begin{equation}\label{eq:cochain}
    \cdots\rightarrow\gr^{-2}[-2]\xrightarrow{\delta}\gr^{-1}[-1]\xrightarrow{\delta}\gr^0\xrightarrow{\delta}\gr^1[1]\xrightarrow{\delta}\gr^2[2]\rightarrow\cdots
\end{equation}
is a kind of cochain complex object in $\Cscr$.

The filtration $\F^\star$ contains much more information than just the graded
pieces, the differentials $\delta$, and the nullhomotopies. In order to reconstruct more of it,
Nikolaus suggested using the pointed $1$-category $\Xi$ of Joyal~\cite[35.1]{joyal-notes}.
The set of objects is $\bZ_*$, the set of
integers with an extra base point which is both an initial and final object. The
sets of morphisms between integers are
$$\Hom_\Xi(m,n)=\begin{cases}
    \ast&\text{if $n\neq m,m-1$,}\\
    {\id,\ast}&\text{if $n=m$,}\\
    {\delta,\ast}&\text{if $n=m-1$.}
\end{cases}$$
In particular, in $\Xi$, one has $\delta\circ\delta\we\ast$ whenever the
left-hand side makes sense.

\begin{definition}
    Let $\Cscr$ be a pointed $\infty$-category. A coherent chain
    complex in
    $\Cscr$ is a pointed functor $\Xi\rightarrow\Cscr$. The $\infty$-category
    of coherent chain complexes in $\Cscr$ is the functor category
    $\Ch_\bullet(\Cscr)\we\Fun_\ast(\Xi,\Cscr)$. Similarly, a coherent cochain
    complex in $\Cscr$ is a pointed functor $\Xi^\op\rightarrow\Cscr$ and
    the $\infty$-category of coherent chain complexes in $\Cscr$ is
    $\Ch^\bullet(\Cscr)=\Fun_\ast(\Xi^\op,\Cscr)$.
\end{definition}

\begin{example}
    If $\Ascr$ is an abelian category, then
    $\Ch^\bullet(\Ascr)\we\Fun_\ast(\Xi^\op,\Ascr)$ recovers the usual abelian
    category of cochain complexes in $\Ascr$.
\end{example}

The following theorem is due to Ariotta~\cite{ariotta} and can also be
extracted from Raksit~\cite{raksit} or~\cite[Ex.~1.10.6]{amgr4}. There is a connection to
quasicoherent sheaves on $\bA^1/\Gm$; see~\cite{lurie-rotation,moulinos-filtrations}.

\begin{theorem}[Ariotta's $\E^1$-page theorem]\label{thm:ariotta}
    Let $\Cscr$ be a stable $\infty$-category with sequential limits. The
    associated graded functor $\gr^\star\colon\F\Cscr\rightarrow\Gr\Cscr$
    factors through the forgetful functor
    $\Ch^\cs(\Cscr)\rightarrow\Gr\Cscr$ and induces
    an equivalence $\widehat{\F}\Cscr\we\Ch^\cs(\Cscr)$ between the
    $\infty$-category of complete decreasing filtrations in $\Cscr$ and the
    $\infty$-category of coherent cochain complexes in $\Cscr$.
\end{theorem}

The theorem says that giving a filtered object is equivalent to giving the $\E^1$-page of
the spectral sequence associated to the filtration, at least if this
$\E^1$-page is viewed as an appropriately homotopy coherent object. 

\begin{notation}
    If $\F^\fs\in\F\Cscr$ is a filtered object, let $\gr^\cs_\F
    [\cs]$ denote the associated coherent cochain complex as displayed in~\eqref{eq:cochain}. If
    $C^\cs$ is a coherent cochain complex in $\Cscr$, let $\sigma^\fs C$ be the
    associated decreasing filtration in $\Cscr$, obtained via the inverse to the equivalence in
    Theorem~\ref{thm:ariotta}.
\end{notation}

\begin{example}
    If $C^\bullet$ is a cochain complex in $\Mod_{\bZ}$, then $\sigma^\fs C$ is
    the usual stupid filtration on the underlying homotopy type $|C^\bullet|$ of $C$ in
    $\D(\bZ)$ as in Example~\ref{ex:stupid}.
\end{example}

The functor $\widehat{\F}\Cscr\rightarrow\Ch^\bullet(\Cscr)$ is a souped-up
version of Lurie's construction~\eqref{eq:cochain} and indeed the restriction to the
$2$-skeleton of (the nerve of) $\Xi$ recovers the sequence~\eqref{eq:cochain}
together with the nullhomotopies $\delta\circ\delta\we 0$.

The theorem is blind to any preferred
$t$-structure on $\Cscr$. When one is fixed, the induced $t$-structure on $\widehat{\F}\Cscr$
is called the Beilinson $t$-structure.

\begin{definition}[Complete Beilinson $t$-structures]
    If $\Cscr$ admits sequential limits and is equipped with a $t$-structure $(\Cscr_{\geq 0},\Cscr_{\leq
    0})$, then $\Ch^\bullet(\Cscr)$ inherits the pointwise $t$-structure, where
    a coherent cochain complex $X^\bullet$ is connective if and only if each
    $X^n\in\Cscr_{\geq 0}$. The associated $t$-structure obtained through transport via
    Theorem~\ref{thm:ariotta} on the equivalent
    $\infty$-category $\widehat{\F}\Cscr$ is called the Beilinson
    $t$-structure associated to $(\Cscr_{\geq 0},\Cscr_{\leq 0})$. Unwinding
    the equivalence, the connective objects in the Beilinson $t$-structure are
    those complete filtrations $\F^\star$ such that
    $\gr^n_\F\in\Cscr_{\geq -n}$ and the coconnective objects are those complete filtrations $\F^\star$ with
    $\gr^n_\F\in\Cscr_{\leq -n}$ for all $n$. The heart of the Beilinson $t$-structure is
    equivalent to $\Ch^\bullet(\Cscr^\heart)$, the abelian category of cochain
    complexes in the heart of $\Cscr$.
\end{definition}

\begin{definition}[Incomplete Beilinson $t$-structures]\label{def:noncompletebeilinson}
    When $\Cscr$ admits sequential limits, there is a localization sequence
    $\Cscr\rightarrow\F\Cscr\rightarrow\widehat{\F}\Cscr$ where the right
    adjoints are given by taking the limit of a filtration and by including
    complete filtrations into all filtrations, respectively. If a $t$-structure
    on $\Cscr$ is fixed, the induced Beilinson $t$-structure on
    $\widehat{\F}\Cscr$ defines a $t$-structure on $\F\Cscr$, also called the Beilinson
    $t$-structure, by
    declaring that $(\F\Cscr)_{\geq 0}^\B$ is the full subcategory of
    filtrations $\F^\star$
    such that $\gr^n_\F\in\Cscr_{\geq -n}$ for all $n\in\bZ$. Equivalently, $(\F\Cscr)_{\geq 0}^\B$
    consists of those filtrations whose completion is connective in the Beilinson $t$-structure on
    $\widehat{\F}\Cscr$. In particular, the
    full subcategory $\Cscr\subseteq\F\Cscr$ of constant filtrations is
    $\infty$-connective in the (or any) Beilinson $t$-structure on $\F\Cscr$.
    With this definition, $(\F\Cscr)_{\leq 0}^\B\we(\widehat{\F}\Cscr)_{\leq 0}^\B$. In particular,
    any bounded above object is complete and the inclusion functor
    $\widehat{\F}\Cscr\rightarrow\F\Cscr$ is $t$-exact.
\end{definition}

\begin{remark}[Bounded above objects]\label{rem:bounded}
    If $\Cscr$ is a stable $\infty$-category which admits sequential limits and a $t$-structure on
    $\Cscr$ is fixed, then the objects $\F^\star$ of $(\F\Cscr)_{\leq 0}^\B\we(\widehat{\F}\Cscr)_{\leq 0}^\B$
    have the property that the filtration is complete and $\gr^n_\F\in\Cscr_{\leq -n}$. It follows
    inductively that each $\gr^{[n,m)}_\F\in\Cscr_{\leq -n}$ for $m\geq n$ and hence that the limit
    $\lim_m\gr^{[n,m)}_\F$ is in $\Cscr_{\leq -n}$ since the $(-n)$-coconnectives are closed under
    limits. By completeness, $\lim_m\gr^{[n,m)}_\F\we\F^n$, so $\F^n\in\Cscr_{\leq -n}$ for all
    $n$. If the $t$-structure on $\Cscr$ is right separated, then the converse is true: if
    $\F^\star$ is a filtration in $\F\Cscr$ with the property that $\F^n\in\Cscr_{\leq -n}$,
    then it is in $(\F\Cscr)_{\leq 0}$ and, in particular, complete. Indeed, the condition implies
    that each $\gr^n_\F\in\Cscr_{\leq -n}$ by taking cofibers, so it is enough to check completeness. But, $\lim_m\F^m$
    is in $\Cscr_{\leq -n}$ for all $n$, so it vanishes by right separatedness.
\end{remark}

\begin{remark}[Bounded Beilinson $t$-structures]
    If $\Cscr$ is a small stable $\infty$-category with a bounded $t$-structure, then it can be embedded
    into $\Ind(\Cscr)$ and $\Ind(\Cscr)$ admits a $t$-structure which is
    compatible with filtered colimits and right complete. The induced Yoneda embedding
    $\Cscr\rightarrow\Ind(\Cscr)$ is $t$-exact. See for
    example~\cite[Prop.~2.13]{agh}. Of course, $\Ind(\Cscr)$ admits
    sequential limits, so the formalism above can be applied to discuss the Beilinson $t$-structure
    on $\F\Ind(\Cscr)$. However, it may be desirable to consider only a certain class of
    filtrations, $\F^b\Cscr\subseteq\F\Ind(\Cscr)$ consisting of the filtrations $\F^\star M$ in $\F\Cscr$ such
    that $\F^i M\we 0$ for $i$ sufficiently large and $\gr^iM\we 0$ for $i$ sufficiently small.
    (These are the eventually zero, eventually constant filtrations.) The Beilinson $t$-structure
    on $\F\Ind(\Cscr)$ restricts to a $t$-structure on $\F^b\Cscr$ whose heart is the abelian
    category $\Ch^{\bullet,b}(\Cscr^\heart)$ of bounded cochain complexes of objects of
    $\Cscr^\heart$. To see this, note that the truncation functors $\tau_{\geq 0}^\B$ and
    $\tau_{\leq 0}^\B$ on $\F\Ind(\Cscr)$ preserve $\F^b\Cscr$.
\end{remark}

The following construction will place an important role in the analysis of the d\'ecalage functor.

\begin{construction}[Realization of cochain complexes]\label{const:realization}
    Let $\Cscr$ be a stable $\infty$-category admitting sequential limits and colimits with
    a fixed $t$-structure $(\Cscr_{\geq 0},\Cscr_{\leq 0})$. There is an
    induced functor $\Ch^\bullet(\Cscr^\heart)\rightarrow\Cscr$ defined by the composition
    $$\Ch^\bullet(\Cscr^\heart)\hookrightarrow(\widehat{\F}\Cscr)_{\geq
    0}\hookrightarrow\widehat{\F}\Cscr\hookrightarrow\F\Cscr\xrightarrow{|-|}\Cscr.$$
    In other words, this functor takes a cochain complex $C^\cs$ to
    $|\sigma^\fs C|$, the object underlying the stupid filtration
    $\sigma^\fs C$ associated to $C^\cs$. Write $|C^\cs|$ for this underlying
    object.
\end{construction}

\begin{lemma}\label{lem:cohomology_of_complex}
    Let $\Cscr$ be a stable $\infty$-category admitting sequential limits and colimits with
    a fixed $t$-structure $(\Cscr_{\geq 0},\Cscr_{\leq 0})$.
    If $C^\bullet$ is a cochain complex, then there is a natural isomorphism $\pi_n(|C^\bullet|)\iso\H^{-n}(C^\bullet)$.
\end{lemma}

\begin{proof}
    We prove the case $n=0$. There is a natural isomorphism $\pi_0\gr^{[-1,2)}_\sigma C\iso\H^0(C^\bullet)$. Indeed,
    $\pi_0\gr^{[0,2)}_\sigma C$ fits into a short exact sequence
    $$0\rightarrow\pi_0\gr^{[0,2)}_\sigma C\rightarrow C^0\xrightarrow{\d}
    C^1\rightarrow\pi_{-1}\gr^{[0,2)}_\sigma C\rightarrow 0,$$
    so $\pi_0\gr^{[0,2)}_\sigma C$ is identified with $Z^0$, the object of cycles in degree $0$.
    Now, the fiber sequence $\gr^{[0,2)}_\sigma C\rightarrow\gr^{[-1,2)}_\sigma C\rightarrow\gr^{-1}_\sigma C$
    induces an exact sequence $$C^{-1}\xrightarrow{\d}
    Z^0\rightarrow\pi_0\gr^{[-1,2)}_\sigma C\rightarrow 0,$$
    which proves the claim.
    We also have that $\tau_{\geq
    0}\gr_\sigma^{[-1,\infty)}C\we\lim_n\tau_{\geq 0}\gr_\sigma^{[-1,n)}C$ as $\tau_{\geq 0}$
    preserves limits and since the filtration $\sigma^\star C$ is complete.
    However, $\tau_{\geq 0}\gr^{[-1,n+1)}_\sigma C\rightarrow\tau_{\geq 0}\gr^{[-1,n)}_\sigma C$ is
    an equivalence for $n\geq 2$ since it is the fiber of a map $\gr^{[-1,n)}_\sigma C\rightarrow(\gr^{n}_\sigma
    C)[1]$, $\tau_{\geq 0}(\gr^n_\sigma C)[1]\we\tau_{\geq 0}(C^n[-n+1])\we 0$ for $n\geq 2$, and $\tau_{\geq 0}$ preserves
    fibers. It follows that $\pi_0\gr^{[-1,\infty)}_\sigma C\iso\H^0(C^\bullet)$.
    For $m\geq 1$, the map $\tau_{\leq
    0}\sigma^{[-m,\infty)}_\sigma C\rightarrow\tau_{\leq 0}\sigma^{[-m-1,\infty)}_\sigma C$ is an
    equivalence since $\tau_{\leq 0}$ preserves cofibers and this map is a cofiber of a map
    $\gr^{-m-1}_\sigma C[-1]\rightarrow\sigma^{[-m-1,\infty)}_\sigma C$, where $\gr^{-m-1}_\sigma
    C[-1]\we C^{-m-1}[m]$ is connective for $m\geq 1$. As $\tau_{\leq 0}$ commutes with colimits,
    it follows that $\tau_{\leq 0}C\we\colim\tau_{\leq 0}\gr^{[-m,\infty)}_\sigma C\we\tau_{\leq
    0}\gr^{[-1,\infty)}_\sigma C$. This completes the proof.
\end{proof}

\begin{corollary}
    Suppose that $\Cscr$ is a stable $\infty$-category with sequential limits and colimits. If
    $\Cscr$ is equipped with a $t$-structure which is separated, then the functor
    $$\Ch^\bullet(\Cscr^\heart)\xrightarrow{C^\bullet\mapsto|C^\bullet|}\Cscr$$ factors through $\D(\Cscr^\heart)$, the
    $\infty$-categorical localization of $\Ch^\bullet(\Cscr)$ at the quasi-isomorphisms.
\end{corollary}

We conclude this section with some generalities on the properties of Beilinson $t$-structures on
$\widehat{\F}\Cscr$ as inherited from a $t$-structure on $\Cscr$. This material will not be used
below. For background on prestable $\infty$-categories and the
terminology used below, see~\cite[App.~C]{sag}.

\begin{proposition}\label{prop:inheritance}
    Let $I$ be a $1$-category and let $\Cscr$ be a stable $\infty$-category
    equipped with a $t$-structure $(\Cscr_{\geq 0},\Cscr_{\leq 0})$. The
    following properties of $(\Cscr_{\geq 0},\Cscr_{\leq 0})$ are inherited by $\Fun(I,\Cscr)$ and, if $I$ is
    pointed, they are inherited by $\Fun_*(I,\Cscr)$:
    \begin{enumerate}
        \item[{\em (a)}] right or left separatedness;
        \item[{\em (b)}] right or left completeness;
        \item[{\em (c)}] compatibility with filtered colimits;
        \item[{\em (d)}] Grothendieck;
        \item[{\em (e)}] compactly generation.
    \end{enumerate}
    If, moreover, $I$ has the property that $\Hom_I(i,i)=\{\id_i\}$ for all $i\in I$,
    then the following property is inherited by $\Fun(I,\Cscr)$ and $\Fun_*(I,\Cscr)$:
    \begin{enumerate}
        \item[{\em (f)}] compatibility with filtered colimits and either weakly $n$-complicial or
            $n$-complicial.
    \end{enumerate}
\end{proposition}

\begin{proof}
    For part (a), if $X\colon I\rightarrow\Cscr$ is an object of
    $\Fun(I,\Cscr)$ which is in $\cap_n\Fun(I,\Cscr)_{\leq
    n}\we\cap_n\Fun(I,\Cscr_{\leq n})$, then it must vanish if $\Cscr$ is
    right separated. Similarly for left separatedness.

    If $\Cscr\we\lim_n\Cscr_{\leq n}$, then
    $\Fun(I,\Cscr)\we\lim_n\Fun(I,\Cscr_{\leq n})\we\lim_n\Fun(I,\Cscr)_{\leq n}$,
    so $\Cscr$ is left complete. Similarly for right completeness. This gives
    (b).

    If $\Cscr_{\leq 0}$ is closed under filtered colimits in $\Cscr$, then the
    same is true for $\Fun(I,\Cscr)_{\leq 0}\we\Fun(I,\Cscr_{\leq 0})$ in
    $\Fun(I,\Cscr)$ as colimits in a functor category are computed pointwise.
    This proves (c).

    Since the $t$-structure on $\Cscr$ is Grothendieck, it is right complete
    and compatible with filtered colimits. Moreover, $\Cscr$ and $\Cscr_{\geq
    0}$ are presentable. These are inherited by $\Fun(I,\Cscr)$, which shows
    part (d).

    For part (e), if $X\in\Cscr$ is compact, then $i_!X$ is compact in
    $\Fun(I,\Cscr)$ for any $i\in I$ since $i^*$ preserves filtered colimits.
    The collection of $i_!X$ as $i$ ranges over $I$ and $X$ ranges over a set
    of compact generators of $\Cscr$ provides a set of compact generators for
    $\Fun(I,\Cscr)$.

    Let $X\colon I\rightarrow\Cscr_{\geq 0}$ be an object of
    $\Fun(I,\Cscr)_{\geq 0}$ and assume that $\Cscr$ is $n$-complicial.
    For each object $X(i)\we i^*X\in\Cscr_{\geq 0}$, this means that there is a
    morphism $y_i\rightarrow X(i)$ inducing a surjection on $\pi_0$ in
    $\Cscr^\heart$ and such that $y_i\in\Cscr_{[0,n]}$. By adjunction, there is a map $i_!Y_i\rightarrow X$ in
    $\Fun(I,\Cscr)$. Computing the value of $i_!Y_i$ at $j\in I$ yields
    $$\bigoplus_{f\in\Hom_I(i,j)}Y_i$$ in $\Cscr$ by the assumption on $I$, which is in $\Cscr_{[0,n]}$
    since the $t$-structure is compatible with filtered colimits. In particular, $i_!Y_i$ is in
    $\Fun(I,\Cscr)_{[0,n]}$ and $i_!Y_i\rightarrow X(i)$ is surjective on
    $i^*\pi_0$. Taking a sum over $i\in I$ of the $i_!Y_i\rightarrow X$ yields
    $$\bigoplus_{i\in I}i_!Y_i\rightarrow X.$$ The map is surjective on $\pi_0$
    and the sum is in $\Fun(I,\Cscr)_{[0,n]}$ by compatibility of the
    $t$-structure on $\Cscr$ with filtered colimits. This proves $n$-complicialness.
    The proof for weak $n$-compliciality is the same, so (f) is proved.
\end{proof}

\begin{corollary}
    Let $\Ascr$ be a Grothendieck abelian category. The $\infty$-category
    $\FDhat(\Ascr)\we\Ch^\bullet(\D(\Ascr))$ is Grothendieck and
    $0$-complicial: the natural map
    $\D(\Ch^\bullet(\Ascr))\rightarrow\FDhat(\Ascr)\we\Ch^\bullet(\D(\Ascr))$ is an equivalence.
\end{corollary}

\begin{example}
    $\FDhat(\bZ)\we\Ch^\bullet(\D(\bZ))\we\D(\Ch^\bullet(\bZ))$.
\end{example}

\begin{remark}
    Let $\Ascr$ be a Grothendieck abelian $n$-category. There is a canonical
    separated $n$-complicial Grothendieck prestable $\infty$-category
    $\D(\Ascr)_{\geq 0}$ associated to $\Ascr$ by~\cite[Prop.~C.5.4.5]{sag}
    with presentable stabilization $\D(\Ascr)$. By construction,
    $\D(\Ascr)_{[0,n-1]}\we\Ascr$. By
    Proposition~\ref{prop:inheritance}, the $\infty$-category
    $\FDhat(\D(\Ascr))\we\Ch^\bullet(\D(\Ascr))$ is equivalent to
    $\D(\Ch^\bullet(\Ascr))$, the derived $\infty$-category of the Grothendieck
    abelian $n$-category $\Ch^\bullet(\Ascr)$.
\end{remark}

\section{D\'ecalage}\label{sec:decalage}

This section contains the definition of a spectral sequence in an abelian category as well as two
constructions of the spectral sequence associated to a filtered object in a stable
$\infty$-category equipped with a $t$-structure, one via d\'ecalage and one following
Lurie~\cite[Prop.~1.2.2.7]{ha}. The main theorem is that the definitions agree.

\begin{definition}[Spectral sequences]\label{def:ss}
    Let $\Ascr$ be an abelian category. A spectral sequence in $\Ascr$ starting at the $a$th page consists of
    \begin{enumerate}
        \item[(i)] objects $\E^r_{s,t}$ of $\Ascr$ for $r\geq a$ and $s,t\in\bZ$,
        \item[(ii)] differentials $\d^r_{s,t}\colon\E^r_{s,t}\rightarrow\E^r_{s-r,t+r-1}$ for
            $r\geq a$ and $s,t\in\bZ$, and
        \item[(iii)] isomorphisms between the cohomology at $\E^r_{s,t}$ with respect to
            $\d^r$ and $\E^{r+1}_{s,t}$ for $r\geq a$ and $s,t\in\bZ$.\footnote{Some
            sources, such as~\cite{mccleary}, require only that the $\E^{r+1}$-page be isomorphic
            to the cohomology of the $\E^r$-page, without fixing the isomorphism.}
    \end{enumerate}
    For fixed $r\geq a$, the collection $\{\E^r_{s,t}\}$ together with the differentials $\{\d^r_{s,t}\}$
    form the $r$th page of the spectral sequence.
\end{definition}

\begin{remark}
    In this paper, most spectral sequences will start at the $1$st page. See
    Section~\ref{sec:reindexing} for comments about alternative indexing conventions.
\end{remark}

The primary sources of spectral sequences are strictly filtered complexes (see for
example~\cite{mccleary,weibel-homological}), exact couples (see for example~\cite[Sec.~5.9]{weibel-homological}), and
filtrations in stable $\infty$-categories with fixed $t$-structures.
Note that while part (iii) of Definition~\ref{def:ss}
implies that the $\E^r$-page of a spectral sequence determines the objects on the $\E^{r+1}$-page,
it does not generally determine the $\d^{r+1}$-differential. This must be constructed in some way.

Fix a stable $\infty$-category $\Cscr$ with sequential limits and colimits and let $\F^\star$ be
a decreasing filtration in $\Cscr$. Assume also that $\Cscr$
is equipped with a $t$-structure. As mentioned above, the basic problem is to go from
presumably accessible information about the graded pieces, specifically the homotopy objects
$$\pi_*\gr^\star_\F,$$ 
and to learn something about $\pi_*\F^{-\infty}$.
There are two fundamental obstructions to a class $x\in\pi_{i}\gr^{j}_\F$
contributing to $\pi_{i}\F^{-\infty}$: the obstructions to lifting
$x$ to $\pi_i\F^{j}$ and then the possibility that a lift
$\tilde{x}\in\pi_{i}\F^{j}$ is in the kernel of the map
$$\pi_{i}\F^{j}\rightarrow\pi_{i}\F^{-\infty}.$$
A spectral sequence is a way of packaging these wrinkles together in a
reasonable format.

\begin{construction}
    Associated to $\F^\star$ is the coherent cochain complex
    $$\cdots\rightarrow\gr^{-s-1}_\F[-s-1]\rightarrow\gr^{-s}_\F[-s]\rightarrow\gr^{-s+1}_\F[-s+1]\rightarrow\cdots$$
    discussed in the previous section. Applying $\pi_t$ with respect to the fixed
    $t$-structure on $\Cscr$ to the coherent cochain
    complex produces a cochain complex in $\Cscr^\heart$ of the form
    $$\cdots\rightarrow\pi_t\gr^{-s-1}_\F[-s-1]\rightarrow\pi_t\gr^{-s}_\F[-s]\rightarrow\pi_t\gr^{-s+1}_\F[-s+1]\rightarrow\cdots,$$
    which can be rewritten as
    $$\cdots\rightarrow\pi_{s+t+1}\gr^{-s-1}_\F\rightarrow\pi_{s+t}\gr^{-s}_\F\rightarrow\pi_{s+t-1}\gr^{-s+1}_\F\rightarrow\cdots.$$
\end{construction}

\begin{definition}[$\E^1$-page]\label{def:e1}
    The $\E^1$-page of the spectral sequence associated to $\F^\star$
    is the graded chain complex $$\E^1_{s,t}=\pi_{s+t}\gr^{-s}M$$ with
    differentials
    $$\d^1_{s,t}\colon\E^1_{s,t}\rightarrow\E^1_{s-1,t}.$$
    In other words, $$\E^1_{\bullet,t}=\pi_t^\B(\F)^{-\bullet},$$
    the chain complex associated to cochain complex $\pi_t^\B(\F)$.
    In particular, $\d^1\circ\d^1=0$, so objects $\E^2_{s,t}(\F)\in\Cscr^\heart$ are defined as the homology of
    the $\E^1$-page with respect to $\d^1$.
\end{definition}

\begin{construction}[D\'ecalage]\label{const:decalage}
    Let $\Cscr$ be a stable $\infty$-category with sequential limits and colimits which
    admits a $t$-structure $(\Cscr_{\geq 0},\Cscr_{\leq 0})$. Let $\F^\star$
    be a filtered object of $\Cscr$. Let $\tau_{\geq\star}^\B(\F)$
    denote the Whitehead tower of $\F^\star$ with respect to the Beilinson
    $t$-structure on $\F\Cscr$ (associated to the fixed $t$-structure on
    $\Cscr$). Note that each $\tau_{\geq n}^\B(\F)$ is itself a
    filtered object of $\Cscr$. By applying realization to the Whitehead tower
    $$\cdots\rightarrow\tau_{\geq n+1}^\B(\F)\rightarrow\tau_{\geq
    n}^\B(\F)\rightarrow\tau_{\geq n-1}^\B(\F)\rightarrow\cdots$$ one obtains a new filtered object
    $$\cdots\rightarrow|\tau_{\geq n+1}^\B(\F)|\rightarrow|\tau_{\geq
    n}^\B(\F)|\rightarrow|\tau_{\geq n-1}^\B(\F)|\rightarrow\cdots$$
    of $\Cscr$, called
    the d\'ecalage of $\F^\star$ and denoted $\Dec(\F)^\star$.
    Since there are natural maps $\tau_{\geq n}^\B(\F)\rightarrow\F^\star$, taking colimits gives natural maps $$|\tau_{\geq
    n}^\B(\F)|\rightarrow|\F^\star|.$$ Thus, if $\F^\star$ is a filtration on $M$, then
    $\Dec(\F)^\star$ can be equipped with the structure of a filtration on $M$ as well.
\end{construction}

\begin{remark}
    D\'ecalage gives an endofunctor of $\F\Cscr$ when $\Cscr$ admits sequential limits and colimits
    and is equipped with a $t$-structure. There is an analogue of d\'ecalage for any sequence of
    full subcategories $\cdots\rightarrow(\F\Cscr)_n\rightarrow(\F\Cscr)_{n-1}\rightarrow\cdots$
    for which the inclusions admits right adjoints. The case here is where
    $(\F\Cscr)_n=(\F\Cscr)_{\geq n}^\B$, the full subcategory of $n$-connective objects with
    respect to the Beilinson $t$-structure.
\end{remark}

\begin{remark}[Graded pieces of d\'ecalage]\label{rem:graded_decalage}
    As $\tau_{\geq n+1}^\B(\F)\rightarrow\tau_{\geq n}^\B(\F)\rightarrow\pi_n^\B(\F)[n]$ forms a cofiber sequence in filtered
    complexes, one finds that the associated graded pieces of $\Dec(\F)^\star$ are given by $$\gr^n_{\Dec(\F)}\we|\pi_n^\B(\F)[n]|\we|\pi_n^\B(\F)|[n].$$
    Under the identification of the heart
    of the Beilinson $t$-structure with cochain complexes, $\pi_n^\B(\F)$ is the cochain complex
    $$\pi_n^\B(\F)^\bullet\colon\cdots\rightarrow\pi_{n+1}\gr^{-1}\rightarrow\pi_n\gr^0\rightarrow\pi_{n-1}\gr^1\rightarrow\pi_{n-2}\gr^2\rightarrow\cdots$$
    where $\pi_n\gr^0$ is placed in cohomological degree $0$.
    Thus, $\gr^n_{\Dec(\F)}$ is the $n$-fold suspension of the realization of this
    cochain complex in the sense of Construction~\ref{const:realization}.
    After reindexing to a chain complex, this complex is precisely the $n$th horizontal line
    on the $\E^1$-page of the spectral sequence of $\F^\star$.
\end{remark}

\begin{definition}[D\'ecalage definition of the spectral sequence]\label{def:ss_decalage}
    Let $\Cscr$ be a stable $\infty$-category which admits sequential limits and colimits and
    suppose $\Cscr$ is equipped with a fixed $t$-structure. If $\F^\star\in\F\Cscr$ is a filtered
    object, then let $\Dec^{(r)}(\F)^\star$ denote the $r$-fold composition of the d\'ecalage
    functor for $r\geq 0$. Let $$\E^{r+1}_{s,t}(\F)=\E^1_{-(r-1)s-rt,rs+(r+1)t}(\Dec^{(r)}(\F))$$
    with the corresponding differentials.
\end{definition}

\begin{lemma}
    Suppose that $\Cscr$ is a stable $\infty$-category with sequential limits and colimits and
    equipped with a $t$-structure. If $\F^\star\in\F\Cscr$, then the pages
    of Definition~\ref{def:ss_decalage} admit the structure of a spectral sequence starting at the $1$st page.
\end{lemma}

\begin{proof}
    The bidegree of the differential on $\E^1(\Dec^{(r)}(\F))$ is $(-1,0)$. To transform these
    coordinates to those of $\E^{r+1}(\F)$ one uses the inverse to the matrix
    $$\begin{pmatrix}-r+1&-r\\r&r+1\end{pmatrix},$$ which is
        $$\begin{pmatrix}r+1&r\\-r&-r+1\end{pmatrix}.$$ Thus, the reindexed differential on
            $\E^{r+1}_{s,t}$ has bidegree
    $(-r-1,r)$, which is the bidegree of the differential on the $\E^{r+1}$-page of a spectral
    sequence.  Thus, it suffices to construct an isomorphism between the cohomology on the
    $\E^r$-page and the $\E^{r+1}$-page. Inductively, it is enough to do this for $r=1$. By
    Remark~\ref{rem:graded_decalage},
    $$\E^2_{s,t}(\F)\iso\H^{-s}(\pi_t^\B(\F)^\bullet)\iso\pi_s(|\pi_t^\B(\F)^\bullet|)\iso\pi_{s+t}(|\pi_t^\B(\F)|[t])\iso\E^1_{-t,s+2t}(\Dec(\F)),$$
    where the second isomorphism follows from
    Lemma~\ref{lem:cohomology_of_complex}.
\end{proof}

It follows that Definition~\ref{def:ss_decalage} gives one definition of the spectral sequence of a
filtered object in $\Cscr$, at least if $\Cscr$ admits sequential limits and colimits.
In the remainder of this section, Lurie's definition of a spectral sequence of a filtered object in
$\Cscr$ is given and the two definitions are compared.

\begin{construction}
    Let $\Cscr$ be a stable $\infty$-category with a $t$-structure and let $\F^\star\in\F\Cscr$ be
    a filtration in $\Cscr$. All graded pieces below are constructed with respect to $\F^\star$.
    For $r\geq 1$, and $s\in\bZ$, there is a map
    $$\gr^{[-s,-s+r)}\rightarrow\gr^{[-s-r+1,-s+1)}$$
    induced by the commutative square
    $$\xymatrix{
        \F^{-s+r}\ar[r]\ar[d]&\F^{-s}\ar[d]\\
        \F^{-s+1}\ar[r]&\F^{-s-r+1}
    }$$
    by taking horizontal cofibers. These fit into a natural commutative diagram
    \begin{equation}\label{eq:ss}\begin{gathered}
        \xymatrix{\gr^{[-s+r,-s+2r)}\ar[r]\ar[d]&\gr^{[-s,-s+2r)}\ar[r]\ar[d]&\gr^{[-s,-s+r)}\ar[d]\\
        \gr^{[-s+1,-s+r+1)}\ar[r]&\gr^{[-s-r+1,-s+r+1)}\ar[r]&\gr^{[-s-r+1,-s+1)}}
    \end{gathered}\end{equation}
    of cofiber sequences. For $r\geq 1$ and $s,t\in\bZ$, let
    $$\E^r_{s,t}=\im(\pi_{s+t}\gr^{[-s,-s+r)}\rightarrow\pi_{s+t}\gr^{[-s-r+1,-s+1)})$$
    and consider the commutative diagram
    $$\xymatrix{
    \pi_{s+t}\gr^{[-s,-s+r)}\ar[r]\ar[d]&\E^r_{s,t}\ar[r]\ar@{.>}[d]&\pi_{s+t}\gr^{[-s-r+1,-s+1)}\ar[d]\\
    \pi_{s+t-1}\gr^{[-s+r,-s+2r)}\ar[r]&\E^r_{s-r,t+r-1}\ar[r]&\pi_{s+t-1}\gr^{[-s+1,-s+r+1)}
    }$$
    where the rows are epi-mono factorizations and the left and right vertical maps come
    from the boundary maps in homotopy associated to the fiber sequences
    in~\eqref{eq:ss}.
    The claim is that the outer square above commutes. This follows immediately
    from the fact that there is a canonical map from the top fiber sequence above to the
    bottom fiber sequence.
    By functoriality of epi-mono factorization, the dotted arrow exists. Denote
    it by $\d^r_{s,t}$.
\end{construction}

\begin{definition}[Lurie's definition of the spectral sequence]\label{def:er}
    The bigraded collection $\E^r_{s,t}$ of objects of $\Cscr^\heart$
    together with the differentials
    $\d^r_{s,t}\colon\E^r_{s,t}\rightarrow\E^r_{s-r,t+r-1}$ is the $\E^r$-page of the spectral sequence associated to $\F^\star$.
\end{definition}

Parts (ii) and (iii) of the following proposition appears as~\cite[Prop.~1.2.2.7]{ha}.

\begin{proposition}\label{prop:lurie}
    Let $\F^\star M$ be a filtered object in a stable $\infty$-category equipped with a
    $t$-structure.
    \begin{enumerate}
        \item[{\em (i)}] The definitions of the $\E^1$-page given in
            Definitions~\ref{def:e1} and~\ref{def:er} agree.
        \item[{\em (ii)}] The composition $\d^r\circ\d^r$ vanishes for $r\geq 1$.
        \item[{\em (iii)}] The homology of the chain complexes on the
            $\E^r$-page is naturally isomorphic to the terms of the
            $\E^{r+1}$-page for $r\geq 1$.
    \end{enumerate}
    In particular, the sequence of pages of Definition~\ref{def:er} defines a spectral sequence in
    the sense of Definition~\ref{def:ss}.
\end{proposition}

\begin{proof}[Proof of {\em (i)}]
    Part (i) follows by observing that in the second definition
    $\E^1_{s,t}=\pi_{s+t}\gr^{-s}$, just as in the first definition. The check
    that the differentials agree follows from examining the boundary maps
    associated to the fiber sequences in~\eqref{eq:ss}.
\end{proof}

The proof of part (iii) of Proposition~\ref{prop:lurie} is nontrivial. It will follow from
Theorem~\ref{thm:decalage} that Lurie's definition of the $\E^r$-pages agrees with the d\'ecalage
definition of the $\E^r$-page given here. Inductively, this implies parts (ii) and (iii) of
Proposition~\ref{prop:lurie}. In particular, the proof of Theorem~\ref{thm:decalage} does not use
parts (ii) or (iii) of Proposition~\ref{prop:lurie}.

Until the end of the section, unless specified otherwise, $\E^r_{s,t}(\F)$ refers to Lurie's
definition.

\begin{theorem}[Comparison theorem]\label{thm:decalage}
    Let $\Cscr$ be a stable $\infty$-category with sequential limits and colimits and a
    $t$-structure.
    For $r\geq 1$, there are natural isomorphisms $\E^r_{-t,s+2t}(\Dec(\F))\iso\E^{r+1}_{s,t}(\F)$
    compatible with the differentials on the $\E^r(\Dec(\F))$ and $\E^{r+1}(\F)$-pages.
\end{theorem}

\begin{corollary}\label{cor:decalage}
    Let $\Cscr$ be a stable $\infty$-category with sequential limits and colimits and a
    $t$-structure. For $r\geq 0$, there are natural isomorphisms
    $\E^1_{-(r-1)s-rt,rs+(r+1)t}(\Dec^{(r)}(\F))\iso\E^{r+1}_{s,t}(\F)$ compatible with the
    differentials on the $\E^1(\Dec^{(r)}(\F))$ and $\E^{r+1}(\F)$-pages.
\end{corollary}

\begin{remark}
    In particular, the d\'ecalage definition of the spectral sequence of a filtration agrees with
    Lurie's.
\end{remark}

\begin{remark}[Bounded d\'ecalage]
    The main situation we foresee where the hypotheses of Theorem~\ref{thm:decalage} are not met is
    when $\Cscr$ is a small abelian category admitting a bounded $t$-structure, such as $\D^b(X)$
    for a regular noetherian scheme. However, as in Remark~\ref{rem:bounded}, we can embed $\Cscr$
    inside $\Ind(\Cscr)$ and apply the theory of the Beilinson $t$-structure and d\'ecalage to
    $\F\Ind(\Cscr)$. As the $\Dec$ endofunctor preserves $\F^b\Cscr\subseteq\F\Ind(\Cscr)$, we see that
    Theorem~\ref{thm:decalage} also applies to $\F^b\Cscr$.
\end{remark}

The proof is given at the end of the section after some preliminaries.

\begin{example}
    If $\F^\star M$ is the constant filtration on $M$, then each $\tau_{\geq
    n}^\B(\F^\star M)$ is the constant filtration on $M$ and $\Dec(\F)^\star M$
    is the constant filtration on $M$.
\end{example}

\begin{lemma}\label{lem:truncationscomplete}
    If $\F^\star M$ is a complete filtration, then each $\tau_{\geq
    n}^\B(\F^\star M)$ is complete.
\end{lemma}

\begin{proof}
    There is a cofiber sequence $$\tau_{\geq n}^\B(\F^\star
    M)\rightarrow\F^\star M\rightarrow\tau_{\leq n-1}^\B(\F^\star M)$$ of
    filtered objects of $\Cscr$. As $\tau_{\leq n-1}^\B(\F^\star M)$ is
    necessarily complete (see Definition~\ref{def:noncompletebeilinson}), it follows that $\tau_{\geq n}^\B(\F^\star M)$ is
    complete.
\end{proof}

\begin{remark}\label{rem:no_short_exact}
    If $\F^\star M'\rightarrow\F^\star M\rightarrow\F^\star M''$ is a fiber sequence of fibrations,
    so that each $\F^sM'\rightarrow\F^s M\rightarrow\F^s M''$ is a fiber sequence, there is a
    long-exact sequence $$\cdots\rightarrow\pi_{n+1}^\B(\F^\star M'')\rightarrow\pi_{n}^\B(\F^\star
    M')\rightarrow\pi_{n}^\B(\F^\star M)\rightarrow\pi_{n}^\B(\F^\star M'')\rightarrow\cdots$$ 
    of cochain complexes. It does not typically break up into short exact sequences, which explains
    why one does not often find ``short exact sequences of spectral sequences'' in nature.
\end{remark}

There is an important class of examples which run counter to the general situation of
Remark~\ref{rem:no_short_exact} and which will be
important in the proof of Theorem~\ref{thm:decalage}. First, we need a standard lemma.

\begin{lemma}\label{lem:truncation_exact}
    Let $\Cscr$ be a stable $\infty$-category with a $t$-structure $(\Cscr_{\geq 0},\Cscr_{\leq
    0})$. Suppose that $x\rightarrow y\rightarrow z$ is a cofiber sequence in $\Cscr$.
    \begin{enumerate}
        \item[{\em (i)}] If
        $\pi_0y\rightarrow\pi_0z$ is an epimorphism in $\Cscr^\heart$, then $\tau_{\geq
            0}x\rightarrow\tau_{\geq 0}y\rightarrow\tau_{\geq 0}z$ is a cofiber sequence in $\Cscr$.
        \item[{\em (ii)}] If $\pi_0x\rightarrow\pi_0y$ is a monomorphism in $\Cscr^\heart$, then
            $\tau_{\leq 0}x\rightarrow\tau_{\leq 0}y\rightarrow\tau_{\leq 0}z$ is a cofiber sequence
            in $\Cscr$.
    \end{enumerate}
\end{lemma}

\begin{proof}
    It suffices to prove part (i). Consider the induced map $z\rightarrow x[1]$ and note that the
    composition $\tau_{\geq 0}z\rightarrow \tau_{\geq 0}x[1]\rightarrow\pi_{-1}x$ factors through the boundary map
    $\pi_0z\rightarrow\pi_{-1}x$, which is assumed to be zero. Thus, there is a lift
    $\tau_{\geq 0}z\rightarrow(\tau_{\geq 0}x)[1]$. As $\Cscr_{\geq 0}$ is prestable, there 
    is by definition a fiber $w$ of this map in $\Cscr_{\geq 0}$ such that the square
    $$\xymatrix{
        w\ar[r]\ar[d]&\tau_{\geq 0}z\ar[d]\\
        0\ar[r]&(\tau_{\geq 0}x)[1]
    }$$ is a fiber square and a cofiber square in $\Cscr_{\geq 0}$.
    As $\tau_{\geq 0}$ preserves fiber squares, we see that there is a map $w\rightarrow\tau_{\geq
    0}y$ in $\Cscr_{\geq 0}$ with trivial fiber and which induces an isomorphism on $\pi_0$. If this map is an equivalence, then $\tau_{\geq
    0}y\rightarrow\tau_{\geq 0}z\rightarrow(\tau_{\geq 0}x)[1]$ is a cofiber sequence in $\Cscr$,
    proving (i). Now, suppose that $a\rightarrow b$ is a map in $\Cscr_{\geq 0}$ which has trivial
    fiber and induces a surjection on $\pi_0$. Let $f$ be the fiber of $a\rightarrow b$ in
    $\Cscr$, so that $\tau_{\geq 0}f\we 0$ is the fiber computed in $\Cscr_{\geq 0}$.
    We have a fiber sequence $a\rightarrow b\rightarrow f[1]$ and hence a fiber sequence
    $a\rightarrow b\rightarrow\tau_{\geq 0}(f[1])$ in $\Cscr_{\geq 0}$. As $\tau_{\geq
    0}(f[1])\iso\pi_{-1}f$ and as $\pi_0a\rightarrow\pi_0b$ is surjective, with have that
    $b\rightarrow\tau_{\geq 0}(f[1])$ is zero, so that $a\we b\oplus\Omega(\pi_{-1}f)\we b$, as
    desired.
\end{proof}

\begin{example}\label{ex:short_exact}
    Suppose that $\F^\star$ is a filtered object in $\Cscr$ and consider the fiber sequence
    $\gr^{[b,c)}\rightarrow\gr^{[a,c)}\rightarrow\gr^{[a,b)}$ of filtered objects using the
    filtrations of Construction~\ref{const:filtered_gradeds} for $a\leq b\leq c$ (including infinite values). Then,
    $$0\rightarrow\pi_0^\B(\gr^{[b,c)})\rightarrow\pi_0^\B(\gr^{[a,c)})\rightarrow\pi_0^\B(\gr^{[a,b)})\rightarrow 0$$
    is short exact. Indeed, the middle term is the cochain complex
    $$\cdots\rightarrow 0\rightarrow\pi_{-a}\gr^a\rightarrow\pi_{-a-1}\gr^{a+1}\rightarrow\cdots\pi_{-c+1}\gr^{c-1}\rightarrow
    0\rightarrow\cdots,$$
    and the left-hand term is the inclusion of the stupid truncation $\sigma^{\geq b}$ with
    quotient the right-hand term the stupid truncation $\sigma^{\leq b-1}$.
    It follows by Lemma~\ref{lem:truncation_exact} that $\Dec(\gr^{[b,c)})\rightarrow\Dec(\gr^{[a,c)})\rightarrow\Dec(\gr^{[a,b)})$
    is in fact a fiber sequence of filtered spectra.
\end{example}

Recall that $\ins^sM$ denotes the filtration on $M$ obtained by left Kan extension of the constant
functor on $M$ along the inclusion $\{s\}\hookrightarrow\bZ^\op$. Concretely, $\F^i\ins^s M\we 0$
for $i>s$ and $\F^i\ins^s M\we M$ for $i\leq s$, with all transition maps the identity. The
filtration on $\gr^a$ of Construction~\ref{const:filtered_gradeds} is naturally equivalent
to $\ins^a\gr^a$.

\begin{lemma}\label{lem:ins_t_exact}
    Let $\Cscr$ be a stable $\infty$-category with sequential limits and with a $t$-structure
    $(\Cscr_{\geq 0},\Cscr_{\leq 0})$.
    The functor $\ins^s\colon\Cscr\rightarrow\F\Cscr$ is $t$-exact with respect to the
    $t$-structure $(\Cscr_{\geq
    -s},\Cscr_{\leq -s})$ on $\Cscr$ and the Beilinson
    $t$-structure associated to $(\Cscr_{\geq 0},\Cscr_{\leq 0})$ on $\F\Cscr$.
\end{lemma}

\begin{proof}
    The functor $\ins^s$ lands in $\widehat{\F}\Cscr$, so it is enough to see it is $t$-exact as a functor
    to $\widehat{\F}\Cscr$ by Definition~\ref{def:noncompletebeilinson}. As a coherent cochain complex,
    $\ins^sM$ is simply $M[s]$ in degree $s$
    and $0$ elsewhere. This is a $t$-exact assignment with the given $t$-structures.
\end{proof}

\begin{example}
    The functor $\ins^0\colon\Cscr\rightarrow\F\Cscr$ is $t$-exact.
\end{example}

\begin{lemma}\label{lem:ins_decalage}
    The filtration $\Dec^\star(\gr^sM)$ is equivalent to $\tau_{\geq -s+\star}\gr^s M$.
\end{lemma}

\begin{proof}
    It is enough to consider the case $s=0$, in which case we have that
    $\tau_{\geq\star}^\B(\gr^0M)=\tau_{\geq\star}^\B(\ins^0M)\we\ins^0(\tau_{\geq\star}M)$ by
    Lemma~\ref{lem:ins_t_exact}. Taking
    colimits gives the result.
\end{proof}

\begin{proof}[Proof of Theorem~\ref{thm:decalage}]
    Consider the following commutative diagram
    $$\xymatrix{
        \pi_0\gr^{[0,r+1)}_\F
        \ar[d]&\pi_0\gr^{[0,\infty)}_\Dec\gr^{[0,r+1)}_\F\ar[r]^{\mathbf{2}}_\twoheadrightarrow\ar[d]\ar[l]_{\mathbf{1}}^\iso&\pi_0\gr^{[0,r)}_\Dec\gr^{[0,r+1)}_\F\ar[d]&\pi_0\gr^{[0,r)}_\Dec\gr^{[0,\infty)}_\F\ar[l]^\iso_{\mathbf{3}}\ar[r]^{\mathbf{4}}_\twoheadrightarrow\ar[d]&\pi_0\gr^{[0,r)}_\Dec\ar[d]\\
        \pi_0\gr^{[-r,1)}_\F&\pi_0\gr^{[-r+1,\infty)}_\Dec\gr^{[-r,1)}_\F\ar[l]_{\mathbf{5}}^\hookleftarrow\ar[r]^{\mathbf{6}}_\iso&\pi_0\gr^{[-r+1,1)}_\Dec\gr^{[-r,1)}_\F&\pi_0\gr^{[-r+1,1)}_\Dec\gr^{[-r,\infty)}_\F\ar[r]^{\mathbf{8}}_\iso\ar[l]_{\mathbf{7}}^\hookleftarrow&\pi_0\gr^{[-r+1,1)}_\Dec,
    }$$ referred to below as the $2\times 5$-diagram. To make precise what is meant here, for example
    $\gr^{[0,\infty)}_\Dec\gr^{[0,r+1)}_\F=\Dec^0(\gr^{[0,r+1)}_\F)$, where $\gr^{[0,r+1)}_\F$ is
    viewed as a filtration via Construction~\ref{const:filtered_gradeds}.
    Similarly, $\gr^{[0,r)}_\Dec$ is the $[0,r)$-associated graded of the d\'ecalage filtration associated
    to the original filtration $\F^\star$. The maps arise from functoriality of the d\'ecalage
    construction.

    Note that $\E^{r+1}_{0,0}(\F)$ is the epi-mono factorization of the left-hand map, while
    $\E^r_{0,0}(\Dec(\F))$ is the epi-mono factorization of the right-hand map. Justification of the
    displayed properties of maps $\mathbf{1},\ldots,\mathbf{8}$ will show
    that there is a canonical isomorphism of these objects of $\Cscr^\heart$.
    \begin{enumerate}
        \item[({\bf 1})] There is a fiber sequence of filtered objects
            $\tau_{\geq 0}^\B\gr^{[0,r+1)}_\F\rightarrow\gr^{[0,r+1)}_\F\rightarrow\tau_{\leq
            -1}^\B\gr^{[0,r+1)}_\F$ whose colimit is
            $\gr^{[0,\infty)}_\Dec\gr^{[0,r+1)}_\F\rightarrow\gr^{[0,r+1)}_\F\rightarrow\gr^{(-\infty,-1]}_\Dec\gr^{[0,r+1)}_\F$.
            By Example~\ref{ex:short_exact}, $\gr^{(-\infty,-1]}_\Dec\gr^{[0,r+1)}_\F$ has a finite
            filtration (induced by $\F$) with
            associated graded pieces $\gr^{(-\infty,-1]}_\Dec\gr^a_\F$. By Lemma~\ref{lem:ins_decalage},
            $\gr^{(-\infty,-1]}_\Dec\gr^a_\F\we\tau_{\leq -a-1}\gr^a_\F$ for $0\leq a\leq r$.
            Thus,
            $\gr^{(-\infty,-1]}_\Dec\gr^{[0,r+1)}_\F$ is an iterated extension of objects in
            $\Cscr_{\leq -1}$, so it is in $\Cscr_{\leq -1}$, which shows that $\mathbf{1}$ is an isomorphism.
        \item[({\bf 2})] By Lemma~\ref{ex:short_exact}, $\gr^{[0,\infty)}_\Dec\gr^{[0,r+1)}_\F$ has a
            finite filtration with graded pieces $\gr^{[0,\infty)}_\Dec\gr^a_\F\we\tau_{\geq
            -a}\gr^a_F$ for $0\leq a\leq r$
            and similarly $\gr^{[0,r)}_\Dec\gr^{[0,r+1)}_\F$ has a finite filtration with graded
            pieces $\gr^{[0,r)}_\Dec\gr^a_\F\we\tau_{[-a,-a+r)}\gr^a_\F$. The natural map is
            compatible with these filtrations, so on its fiber there is a finite filtration with
            graded pieces $\tau_{\geq -a+r}\gr^a_\F$ for $0\leq a\leq r$. It follows that the fiber
            is connective and hence the displayed map $\mathbf{2}$ on $\pi_0$ is surjective.
        \item[({\bf 3})] The fiber of the morphism which induces $\mathbf{3}$ is
            $\gr^{[0,r)}_\Dec\gr^{[r+1,\infty)}_\F$, which has a finite filtration (now using
            $\Dec$) whose associated graded pieces are given by $|\pi_a^\B\gr^{[r+1,\infty)}_\F[a]|$ for $0\leq a<r$. The
            question of whether
            $\mathbf{3}$ is an isomorphism depends only on the homotopy objects of
            $|\pi_a^\B\gr^{[r+1,\infty)}_\F[a]|$ for $0\leq a<r$ and hence the cohomology groups of the
            corresponding cochain complexes thanks to Lemma~\ref{lem:cohomology_of_complex}.
            The cochain complex representing $\pi_a^\B\gr^{[r+1,\infty)}_\F$ is of the form
            $$\cdots\rightarrow
            0\rightarrow\pi_{-r-1+a}\gr^{r+1}_\F\rightarrow\pi_{-r-a+a-1}\gr^{r+2}_\F\rightarrow\cdots,$$
            where $\pi_{-r-1+a}\gr^{r+1}$ has cohomological degree $r+1$. This cochain complex has
            $\H^a=0$ for $a<r+1$ and hence the object $|\pi_a^\B\gr^{[r+1,\infty)}_\F[a]|$ has
            $\pi_b=0$ for $b>-r-1+a$. In particular, as $a<r$, $\pi_b=0$ for $b>-2$.
            Hence, $\mathbf{3}$ is an isomorphism.
        \item[({\bf 4})] The map $\mathbf{4}$ is induced by a map
            $\gr^{[0,r)}_\Dec\gr^{[0,\infty)}_\F\rightarrow\gr^{[0,r)}_\Dec$, whose cofiber is
            $\gr^{[0,r)}_\Dec\gr^{(-\infty,-1]}_\F$, where $\gr^{(-\infty,-1]}_\F$ denotes the
            filtered object
            $$\cdots\rightarrow 0\rightarrow\tfrac{\F^{-1}}{\F^0}\rightarrow\tfrac{\F^{-2}}{\F^0}\rightarrow\cdots.$$ It is enough to show that $\pi_0$ of this cofiber
            vanishes. For this, one can again filter using the Beilinson Whitehead tower and reduce to considering
            $|\pi_a^\B\gr^{(-\infty,-1]}_\F[a]|$ for $0\leq a<r$.
            The cochain complex representing $\pi_a^\B\gr^{(-\infty,-1]}_\F$ is of the form
            $$\cdots\pi_{a+3}\gr^{-3}_\F\rightarrow\pi_{a+2}\gr^{-2}_\F\rightarrow\pi_{a+1}\gr^{-1}_\F\rightarrow
            0\rightarrow\cdots,$$
            where $\pi_{a+1}\gr^{-1}_\F$ is in cohomological degree $-1$. It follows that 
            $|\pi_a^\B\gr^{(-\infty,-1]}_\F[a]|$ has homotopy concentrated in degrees $\geq 1+a$. 
            Thus, $\gr^{[0,r)}_\Dec\gr^{(-\infty,-1]}_\F$ has vanishing $\pi_0$ since
            $0\leq a<r$. In particular, $\mathbf{4}$ is an epimorphism.
        \item[({\bf 5})] Arguing as in the case of $\mathbf{1}$, one reduces to studying the cofiber
            $\gr_\Dec^{(-\infty,-r]}\gr^{[-r,1)}_\F$ which reduces via a finite filtration to
            $\gr_\Dec^{(-\infty,-r]}\gr^a_\F\we\tau_{\leq -a-r}\gr^a_\F$ for $-r\leq a<1$.
            For these $a$, the truncation $\tau_{\leq -a-r}\gr^a_\F$ is in $\Cscr_{\leq 0}$, so the
            cofiber is coconnective. It follows that $\mathbf{5}$ is a monomorphism.
        \item[({\bf 6})] The fiber of the map inducing $\mathbf{6}$ has a finite filtration
            with associated graded pieces $\gr_\Dec^{[1,\infty)}\gr^a_\F\we\tau_{\geq
            -a+1}\gr^a_\F$ for $-r\leq a < 1$. It follows that the fiber is $1$-connective and
            hence $\mathbf{6}$ is an isomorphism.
        \item[({\bf 7})] The fiber of the map in question is
            $\gr^{[-r+1,1)}_\Dec\gr^{[1,\infty)}_\F$, which has a finite filtration with associated
            graded pieces $|\pi_a^\B\gr^{[1,\infty)}_\F[a]|$ for $-r+1\leq a<1$. The homotopy
            objects $\pi_a^\B\gr^{[1,\infty)}_\F$ are represented by cochain complexes
            $$\cdots\rightarrow
            0\rightarrow\pi_{a-1}\gr^1_\F\rightarrow\pi_{a-2}\gr^2_\F\rightarrow\cdots,$$
            where $\pi_{a-1}\gr^1_\F$ is in cohomological degree $1$. Thus, the objects
            $|\pi_a^\B\gr^{[1,\infty)}_\F[a]|$ are in homological degree $\leq -1+a$. As $a\leq 0$,
            these have no $\pi_0$. Thus, the given map $\mathbf{7}$ is injective.
        \item[({\bf 8})] Finally, the cofiber of $\gr^{[-r,\infty)}_\F\rightarrow\F$ is
            $\gr_\F^{(-\infty,-r)}$. Thus, the cofiber of
            $\gr^{[-r+1,1)}_\Dec\gr^{[-r,\infty)}_\F\rightarrow\gr^{[-r+1,1)}_\Dec$ admits a finite
            filtration with graded pieces given by $|\pi_a^\B\gr_\F^{(-\infty,-r)}[a]|$ for $-r+1\leq a<1$.
            The homotopy objects themselves are represented by objects
            $$\cdots\rightarrow\pi_{r+2+a}\gr^{-r-2}_\F\rightarrow\pi_{r+1+a}\gr^{-r-1}\rightarrow 0\rightarrow\cdots,$$
            where the $\gr^{-r-1}$ term is in cohomological degree $-r-1$.
            Thus, the suspension has homotopy concentrated in degrees $\geq r+1+a$. As $-r+1\leq
            a$, this is in degrees at least $2$. Thus, the cofiber does not contribute to $\pi_0$
            and $\mathbf{8}$ is an isomorphism.
    \end{enumerate}
    This completes the justification of the decorations in the commutative diagram above.

    By functoriality of epi-mono factorizations in an abelian category, each commutative square in
    the commutative diagram above induces an isomorphism of epi-mono factorizations of the
    vertical morphisms. In particular, there is a canonical zigzag $$\E^{r+1}_{0,0}(\F)\leftarrow
    A\rightarrow B\leftarrow C\rightarrow\E^{r}_{0,0}(\Dec(F))$$ where each map is an isomorphism.
    This proves that the homotopy objects on the $\E^{r+1}(\F)$ and $\E^r(\Dec(\F))$ pages agree
    naturally for all filtrations and all $r\geq 1$.

    To finish the proof, it is enough to show that the identifications constructed above are compatible
    with the differentials. It is enough to check this for the differentials emanating from
    $(0,0)$. For this, consider the general form of the $2\times 5$-diagram above,
    which computes the isomorphism between the $\E^{r+1}_{s,t}(\F)$-term and the
    $\E^r_{-t,s+2t}(\Dec(\F))$-term:
    $$\scalebox{.75}{
        \xymatrix{
            \pi_{s+t}\gr^{[-s,-s+r+1)}_\F
            \ar[d]&\pi_{s+t}\gr^{[t,\infty)}_\Dec\gr^{[-s,-s+r+1)}_\F\ar[r]\ar[d]\ar[l]&\pi_{s+t}\gr^{[t,t+r)}_\Dec\gr^{[-s,-s+r+1)}_\F\ar[d]&\pi_{s+t}\gr^{[t,t+r)}_\Dec\gr^{[-s,\infty)}_\F\ar[l]\ar[r]\ar[d]&\pi_{s+t}\gr^{[t,t+r)}_\Dec\ar[d]\\
            \pi_{s+t}\gr^{[-s-r,-s+1)}_\F&\pi_{s+t}\gr^{[t-r+1,\infty)}_\Dec\gr^{[-s-r,-s+1)}_\F\ar[l]\ar[r]&\pi_{s+t}\gr^{[t-r+1,1)}_\Dec\gr^{[-s-r,-s+1)}_\F&\pi_{s+t}\gr^{[t-r+1,1)}_\Dec\gr^{[-s-r,\infty)}_\F\ar[r]\ar[l]&\pi_{s+t}\gr^{[t-r+1,1)}_\Dec.
    }}$$
    We will use the following commutative diagram
    $$\xymatrix{
        \gr^{[r+1,2r+2)}_\F\ar@{=}[d]&
        \gr^{[r,\infty)}_\Dec\gr^{[r+1,2r+2)}_\F\ar[l]\ar[d]\ar[r]&
        \gr^{[r,2r)}_\Dec\gr^{[r+1,2r+2)}_\F\ar[d]&
        \gr^{[r,2r)}_\Dec\gr^{[r+1,\infty)}_\F\ar[l]\ar[d]\ar[r]&
        \gr^{[r,2r)}_\Dec\ar@{=}[d]\\
        \gr^{[r+1,2r+2)}_\F\ar[d]&
        \gr^{[0,\infty)}_\Dec\gr^{[r+1,2r+2)}_\F\ar[l]\ar[d]\ar[r]&
        P\ar[d]&
        \gr^{[r,2r)}_\Dec\gr^{[0,\infty)}_\F\ar[l]\ar[d]\ar[r]&
        \gr^{[r,2r)}_\Dec\ar[d]\\
        \gr^{[0,2r+2)}_\F\ar[d]&
        \gr^{[0,\infty)}_\Dec\gr^{[0,2r+2)}_\F\ar[l]\ar[d]\ar[r]&
        \gr^{[0,2r)}_\Dec\gr^{[0,2r+2)}\ar[d]&
        \gr^{[0,2r)}_\Dec\gr^{[0,\infty)}_\F\ar[l]\ar[d]\ar[r]&
        \gr^{[0,2r)}_\Dec\ar[d]\\
        \gr^{[0,r+1)}_\F&
        \gr^{[0,\infty)}_\Dec\gr^{[0,r+1)}_\F\ar[l]\ar[r]&
        \gr^{[0,r)}_\Dec\gr^{[0,r+1)}_\F&
        \gr^{[0,r)}_\Dec\gr^{[0,\infty)}_\F\ar[l]\ar[r]&
        \gr^{[0,r)}_\Dec.
    }$$
    Here, the bottom $3\times 5$-diagram has exact columns, which defines $P$.
    The map $\gr^{[r,2r)}_\Dec\gr^{[r+1,2r+2)}_\F\rightarrow P$ exists, and fits into such a
    commutative diagram, because the composition
    $$\gr^{[r,2r)}_\Dec\gr^{[r+1,2r+2)}_\F\rightarrow\gr^{[0,2r)}_\Dec\gr^{[0,2r+2)}_\F\rightarrow\gr^{[0,r)}_\Dec\gr^{[0,r+1)}_\F$$
    is nullhomotopic.
    By taking vertical boundaries from $\pi_0$ to $\pi_{-1}$ in the $3\times 5$-diagram and then
    precomposing with the map on $\pi_{-1}$ from the top $2\times 5$-diagram one obtains the
    following diagram:
    $$\scalebox{.67}{
        \xymatrix{
            &
            \pi_0\gr^{[0,r+1)}_\F\ar[dd]\ar@{->>}[ld]&
            \pi_0\gr^{[0,\infty)}_\Dec\gr^{[0,r+1)}_\F\ar[l]^{\iso}\ar[r]_{\twoheadrightarrow}\ar[dd]&
            \pi_0\gr^{[0,r)}_\Dec\gr^{[0,r+1)}_\F\ar[dd]&
            \pi_0\gr^{[0,r)}_\Dec\gr^{[0,\infty)}_\F\ar[l]^\iso\ar[r]_\twoheadrightarrow\ar[dd]&
            \pi_0\gr^{[0,r)}_\Dec\ar[dd]\ar@{->>}[dr]&
            \\
            \E^{r+1}_{0,0}(\F)\ar[dd]_{\d^{r+1}_\F}&&&&&&\E^r_{0,0}(\Dec(\F))\ar[dd]^{\d^r_\Dec}\\
            &
            \pi_{-1}\gr^{[r+1,2r+2)}_\F&
            \pi_{-1}\gr^{[0,\infty)}_\Dec\gr^{[r+1,2r+2)}_\F\ar[l]\ar[r]&
            \pi_{-1}P&
            \pi_{-1}\gr^{[r,2r)}_\Dec\gr^{[0,\infty)}_\F\ar[l]_{\mathbf{B}}^\hookleftarrow\ar[r]&
            \pi_{-1}\gr^{[r,2r)}_\Dec&
            \\
            \E^{r+1}_{-r-1,r}(\F)&&&&&&\E^r_{-r,r-1}(\Dec(\F)).\\
            &
            \pi_{-1}\gr^{[r+1,2r+2)}_\F\ar@{=}[uu]\ar@{->>}[ul]&
            \pi_{-1}\gr^{[r,\infty)}_\Dec\gr^{[r+1,2r+2)}_\F\ar[l]^\iso\ar[uu]^{\mathbf{A}}_\iso\ar[r]_\twoheadrightarrow&
            \pi_{-1}\gr^{[r,2r)}_\Dec\gr^{[r+1,2r+2)}_\F\ar[uu]&
            \pi_{-1}\gr^{[r,2r)}_\Dec\gr^{[r+1,\infty)}_\F\ar[l]^\iso\ar[uu]\ar[r]_\twoheadrightarrow&
            \pi_{-1}\gr^{[r,2r)}_\Dec\ar@{=}[uu]\ar@{->>}[ur]&
        }
    }$$
    The middle $3\times 5$-diagram is commutative by commutativity of the $4\times 5$-diagram
    above. The outer trapezoids are commutative by definition of the the $\d^{r+1}_\F$ and
    $\d^r_{\Dec}$-differentials.
    Now, we claim that the arrow marked $\mathbf{A}$ is an isomorphism and that the arrow marked
    $\mathbf{B}$ is injective.

    To check the claim for $\mathbf{A}$, we can argue as in the case of $\mathbf{2}$ above to observe
    that the cofiber of
    $\gr^{[r,\infty)}_\Dec\gr^{[r+1,2r+2)}_\F\rightarrow\gr^{[0,\infty)}_\Dec\gr^{[r+1,2r+2)}_\F$
    admits a finite filtration with graded pieces given by $|\pi_a^\B\gr^{[r+1,2r+2)}_\F[a]|$ for $0\leq a<r$.
    The cochain complex $\pi_a^\B\gr^{[r+1,2r+2)}_\F$ takes the form
    $$\cdots\rightarrow 0\rightarrow\pi_{-r-1+a}\gr^{r+1}_\F\rightarrow\pi_{-r-2+a}\gr^{r+2}_\F\rightarrow\cdots\rightarrow\pi_{-2r-1+a}\gr^{2r+1}_\F\rightarrow
    0\rightarrow\cdots,$$
    where $\pi_{-r-1+a}\gr^{r+1}_\F$ is in cohomological degree $r+1$. This complex has cohomology
    concentrated in $[r+1,2r+1]$ and hence the $a$-fold suspension of its realization $|\pi_a^\B\gr^{[r+1,2r+2)}_\F[a]|$
    has nonzero homotopy objects in $[-2r-1+a,-r-1+a]$. As $r\geq 1$ and $0\leq a<r$, each graded
    piece is $(-2)$-coconnective.

    For $\mathbf{B}$, consider the induced map on long
    exact sequences in homotopy from the fourth column to the third column in the $4\times
    5$-diagram above. It yields a commutative diagram
    $$\xymatrix{
        \pi_0\gr^{[0,2r)}_\Dec\gr^{[0,\infty)}_\F\ar[r]\ar[d]&\pi_0\gr^{[0,r)}_\Dec\gr^{[0,\infty)}_\F\ar[r]\ar[d]&\pi_{-1}\gr^{[r,2r)}_\Dec\gr^{[0,\infty)}_\F\ar[r]\ar[d]_{\mathbf{B}}&\pi_{-1}\gr^{[0,2r)}_\Dec\gr^{[0,\infty)}_\F\ar[d]\\
        \pi_0\gr^{[0,2r)}_\Dec\gr^{[0,2r+2)}_\F\ar[r]&\pi_0\gr^{[0,r)}_\Dec\gr^{[0,r+1)}_\F\ar[r]&\pi_{-1}P\ar[r]&\pi_{-1}\gr^{[0,2r)}_\Dec\gr^{[0,2r+2)}_\F.
    }$$
    The two left vertical arrows and the rightmost vertical arrow are isomorphisms by now-standard
    arguments, which we omit (but one of them is $\mathbf{3}$). This proves that $\mathbf{B}$ is injective.

    We can now check that
    the two maps from $\pi_0\gr^{[0,\infty)}_\Dec\gr^{[0,r+1)}_\F$ to $\E^r_{-r,r-1}(\Dec(\F))$ are
    equal, one given by composition along the top to $\E^r_{0,0}(\Dec(\F))$ followed by $\d^r_\Dec$
    and the other given by the composition of the boundary to
    $\pi_{-1}\gr^{[0,\infty)}_\Dec\gr^{[r+1,2r+2)}_\F$, the inverse of $A$, and the composition
    along the bottom row to $\E^r_{-r,r-1}(\Dec)$. Indeed, the bottom row and top row paths from
    $\pi_0\gr^{[0,\infty)}_\Dec\gr^{[0,r+1)}_\F$ to $\pi_{-1}P$ agree and then the bottom row and
    top row paths to $\pi_{-1}\gr^{[r,2r)}_\Dec\gr^{[0,\infty)}_\F$ agree by injectivity of
    $\mathbf{B}$. The rest follows by commutativity and this completes the proof.
\end{proof}

\begin{remark}
    The arguments in $\mathbf{1},\ldots,\mathbf{8}$ can be deduced pictorially as follows. We show
    the case of $\mathbf{2}$ for $r=7$.
    Figure~\ref{fig:decalagess} shows the spectral sequence for $\gr^{[0,7)}_\Dec\gr^{[0,8)}_\F$ sitting inside
    the spectral sequence for $\F^\star$ as the filled symbols. The map
    $\gr^{[0,\infty)}_\Dec\gr^{[0,8)}_\F\rightarrow\gr^{[0,7)}_\Dec\gr^{[0,8)}_\F$ cuts away the
    $\star$s and we see that it induces a surjection on $\pi_0$, which the green symbols contribute
    to.

    \begin{sseqdata}[name = decalage, homological Serre grading, classes = {draw = none } ]

    \foreach \x in {0} {
        \foreach \y in {1,...,6} {
            \class["\bullet"](\x,\y)
        }
    }

    \foreach \x in {-1} {
        \foreach \y in {0,2,3,4,5,6} {
            \class["\bullet"](\x,\y)
        }
    }

    \foreach \x in {-2} {
        \foreach \y in {0,1,3,4,5,6} {
            \class["\bullet"](\x,\y)
        }
    }

    \foreach \x in {-3} {
        \foreach \y in {0,1,2,4,5,6} {
            \class["\bullet"](\x,\y)
        }
    }
    \foreach \x in {-4} {
        \foreach \y in {0,1,2,3,5,6} {
            \class["\bullet"](\x,\y)
        }
    }
    \foreach \x in {-5} {
        \foreach \y in {0,1,2,3,4,6} {
            \class["\bullet"](\x,\y)
        }
    }
    \foreach \x in {-6} {
        \foreach \y in {0,1,2,3,4,5} {
            \class["\bullet"](\x,\y)
        }
    }

    \foreach \x in {-7} {
        \foreach \y in {0,1,2,3,4,5,6} {
            \class["\bullet"](\x,\y)
        }
    }

    \foreach \x in {-10,...,-8} {
        \foreach \y in {7, 8} {
            \class["\diamond"](\x,\y)
        }
    }

    \foreach \x in {-10,...,3} {
        \foreach \y in {-2,-1} {
            \class["\diamond"](\x,\y)
        }
    }

     \foreach \x in {1,2,3} {
        \foreach \y in {7, 8} {
            \class["\diamond"](\x,\y)
        }
    }

    \foreach \y in {0,...,6} {
        \foreach \x in {-10,...,-8} {
            \class["\circ"](\x,\y)
        }
        \foreach \x in {1,...,3} {
            \class["\circ"](\x,\y)
        }
    }

    \foreach \x in {-7,...,0} {
        \foreach \y in {8} {
            \class["\star"](\x,\y)
        }
    }

    \foreach \x in {-6,...,0} {
        \class["\color{green}\bullet"](\x,-\x)
    }

    \class["\color{green}\star"](-7,7)

    \foreach \x in {-6,...,0} {
        \foreach \y in {7} {
            \class["\star"](\x,\y)
        }
    }

    \d7(0,0)
    \end{sseqdata}

    \begin{figure}[h]
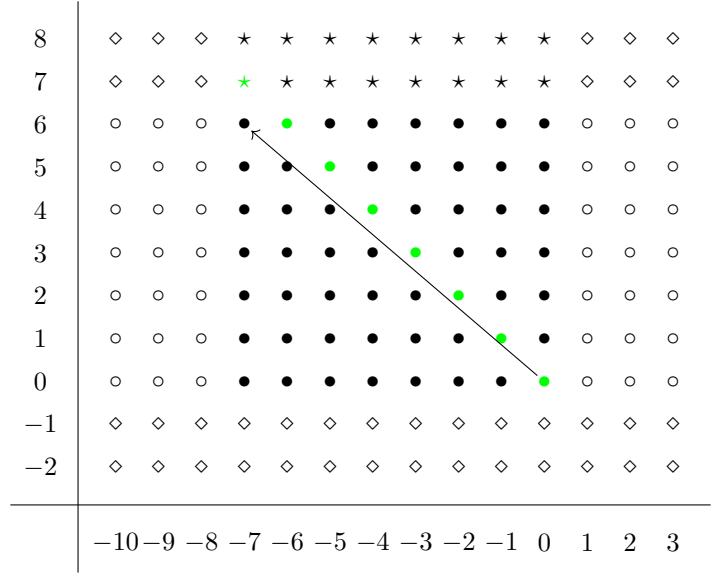

        \centering
        \printpage[name = decalage, page = 7,xscale=0.57,yscale=0.57,x range = {-10}{3},y range = {-2}{8}]
        \caption{The spectral sequence for $\F^\star$ in the proof of
            Theorem~\ref{thm:decalage} in the special case of $r=7$. The solid
            symbols contribute the spectral sequence for $\gr^{[0,7)}_\Dec\gr^{[0,8)}_\F M$, while
            only the solid circles contribute to the spectral squence
            for $\gr^{[0,7)}_{\Dec(\F)}\gr^{[0,8)}_\F$. A $\d^7$-differential is shown and the
            green symbols contribute to $\pi_0$.}
        \label{fig:decalagess}
    \end{figure}
\end{remark}

\section{Deligne's d\'ecalage}

To compare the d\'ecalage functor constructed via the Beilinson $t$-structure
to Deligne's d\'ecalage functor from~\cite[1.3.3]{deligne-hodge-2}, suppose that
$\F^\star M_\cs$ is a filtered chain complex of objects in an abelian category
$\Ascr$ and that the filtration is strict
in the sense that $\F^{s+1} M_n\rightarrow\F^{s}M_n$ is injective for each
$s$ and $n$ in $\bZ$.

\begin{definition}\label{def:delignedecalage}
    Define $\Dec(\F)^\star M_\cs$ by the formula
    $$\Dec(\F)^s M_n=\ker\left(\d\colon\F^{s-n}
    M_{n}\rightarrow\frac{M_{n-1}}{\F^{s-n+1}M_{n-1}}\right).$$
    By construction, the differential on $M_\cs$ restricts to a differential on
    each $\Dec(\F)^s M_\bullet$ so that $\Dec(\F)^\star M_\cs$ defines a filtered
    complex.
\end{definition}

\begin{remark}[Cohomological version]
    If $\F^\star M^\bullet$ is a strictly filtered cochain complex,
    $\Dec(\F)^\star M^\bullet$ is defined analogously as
    $$\Dec(\F)^s
    M^n=\ker\left(\d\colon\F^{s+n}M^n\rightarrow\frac{M^{n+1}}{\F^{s+n+1}M^n}\right).$$
\end{remark}

\begin{construction}
    Each $\Dec(\F)^s M_\bullet$ admits its own filtration by subcomplexes $\G^\star\Dec(\F)^s
    M$ defined by $$\G^w\Dec(\F)^s
    M_n=\ker\left(\d\colon\F^{\max(w,s-n)}M_n\rightarrow\frac{M_{n-1}}{\F^{\max(w,s-n+1)}M_{n-1}}\right).$$
    For $w\geq s-n+1$, $\G^w\Dec(\F)^s M_n=\F^w M_n$ and for $w\leq s-n$,
    $\G^w\Dec(\F)^s M_n=\Dec(\F)^s M_n$. Thus, $\G^\star\Dec(\F)^s M_\cs$ is an
    exhaustive filtration on $\Dec(\F)^s M_\cs$, which is complete if and only
    if the original filtration $\F^\star M$ is. The natural inclusions induce
    maps
    $$\G^\star\Dec(\F)^{s+1}M_\cs\rightarrow\G^\star\Dec(\F)^{s}M_\cs$$
    which make $\G^\star\Dec(\F)^\star M_\cs$ into a (strict) bifiltration
    in chain complexes. For $s$ fixed, $\G^\star\Dec(\F)^s M_\bullet$ is a filtered subcomplex of
    $\F^\star M_\bullet$.
\end{construction}

The following theorem was one of the starting points for this paper.

\begin{theorem}\label{thm:decalagecomparison}
    The natural inclusion $\G^\star\Dec(\F)^s M_\cs\rightarrow\F^\star M_\cs$
    induces an equivalence $\G^\star\Dec(\F)^s
    M\rightarrow\tau_{\geq s}^\B(\F)M$ of towers of filtered objects.
\end{theorem}

\begin{proof}
    It is enough to prove that $\G^\star\Dec(\F)^s M_\cs\rightarrow\F^\star
    M_\cs$ is equivalent to the $s$-connective cover in the Beilinson
    $t$-structure on $\F\D(\bZ)$. For this, it is enough to
    check that $\gr^w_\G\Dec(\F)^s M_\cs\we\tau_{\geq s-w}\gr^w_\F M$ for all $w,s\in\bZ$.
    By construction, $\gr^w_\G\Dec(\F)^s M_\cs$ is given by a complex
    $$\cdots\rightarrow\gr^w_\F M_{s-w+2}\rightarrow\gr^w_\F
    M_{s-w+1}\rightarrow\frac{\ker(\d\colon\F^w
    M_{s-w}\rightarrow\frac{M_{s-w-1}}{\F^{w+1}M_{s-w-1}})}{\F^{w+1}M_{s-w}}\rightarrow
    0\rightarrow\cdots.$$
    The natural map to $\gr^w_\F M_\cs$ is a quasi-isomorphism to $\tau_{\geq
    s-w}\gr^w_\F M_\cs$ since every cycle of $\gr^w_\F M_\cs$ in homological
    degree $s-w$ is represented by an element of $\ker(\d\colon\F^w
    M_{s-w}\rightarrow\frac{M_{s-w-1}}{\F^{w+1}M_{s-w-1}})$.
\end{proof}

\begin{corollary}
    The d\'ecalage construction of Construction~\ref{const:decalage} agrees with
    Deligne's for strictly filtered chain complexes and there is a natural equivalence
    $\G^\star\Dec(\F)^\star M\we\tau_{\geq\star}^\B(\F)M$ of bifiltered objects of $\D(\bZ)$.
\end{corollary}

\section{Convergence}\label{sec:convergence}

As mentioned in the introduction, only a limited convergence result is given.
Hedenlund--Krause--Nikolaus have announced stronger results along the lines of
Boardman's~\cite{boardman} using an $\infty$-categorical approach.

\begin{definition}
    Let $\F^\star M$ be a filtered object of a stable $\infty$-category $\Cscr$
    which is equipped with a fixed $t$-structure. The induced filtration on
    $\pi_mM$ is $$\F^\star\pi_nM=\im(\pi_m\F^\star M\rightarrow\pi_mM).$$
    This is a strict filtration.
\end{definition}

The goal is to compute the associated graded pieces of $\F^\star\pi_nM$ via the spectral sequence
associated to $\F^\star M$.

\begin{definition}
    \begin{enumerate}
        \item[(1)] An upper half-plane spectral sequence is one for which $\E^1_{s,t}=0$ for
            $t<N$ for some fixed $N$. An upper half-plane filtration $\F^\star M$ is one which is
            bounded below in the Beilinson $t$-structure.
            A left half-plane spectral sequence is one for which $\E^1_{s,t}=0$ for
            $s>N$ for some fixed $N$. A left half-plane filtration is a
            filtration $\F^\star M$ such that $\gr^{-s}M\we 0$ for $s>N$.
            Lower and right half-plane filtrations are defined similarly.
        \item[(2)] By intersecting regions in the plane, one finds the notion of first,
            second, third, and fourth quadrant spectral sequences and filtrations. For example, a
            first quadrant filtration is one which is bounded below in the Beilinson
            $t$-structure and where $\gr^{-s}M\we 0$ for $-s$ sufficiently large.
        \item[(3)] Finally, a spectral sequence is column-bounded if it is left and right half-plane,
            so it is concentrated in a bounded vertical strip in the $(s,t)$-plane.
            A filtration $\F^\star M$ is bounded if $\gr^sM\we 0$ for $s\notin[-N,N]$
            for some integer $N$.
            A filtration is row-bounded if it is bounded in the Beilinson
            $t$-structure. A spectral sequence is row-bounded if $\E^1_{s,t}$ vanishes
            for $t\notin[-N,N]$, so the $\E^1$-page
            is concentrated in a bounded horizontal strip in the $(s,t)$-plane.
    \end{enumerate}
\end{definition}

\begin{lemma}
    If $\F^\star M$ is a filtration which is upper half-plane, then the
    associated spectral sequence is upper half-plane. Similarly for left
    half-plane, column-bounded, row-bounded, etc.
\end{lemma}

\begin{proof}
    Left to the reader.
\end{proof}

\begin{notation}[Boundaries and cycles]\label{not:boundaries}
    Each term $\E^r_{s,t}$ is a subquotient of $\pi_{s+t}\gr^{-s}M$, defining
    subobjects $\tilde{\B}^r_{s,t}\subseteq\tilde{\Z}^r_{s,t}$ of
    $\pi_{s+t}\gr^{-s}M$ consisting of boundaries
    and cycles with the
    properties that
    $$\cdots\subseteq\tilde{\B}^r_{s,t}\subseteq\tilde{\B}^{r+1}_{s,t}\subseteq\cdots
    \cdots\subseteq\tilde{\Z}^{r+1}_{s,t}\subseteq\tilde{\Z}^{r}_{s,t}\subseteq\cdots,$$
    and
    $$\frac{\tilde{\Z}^r_{s,t}}{\tilde{\B}^r_{s,t}}\iso\E^{r+1}_{s,t}.$$
    Alternatively, the cycles may be defined as the terms
    $\E^{r+1}_{s,t}(\gr^{[-s,\infty)}_\F)$ (so there are no nonzero differentials hitting the
    $(s,t)$-term); the boundaries may be defined as the terms
    $\E^{r+1}_{s,t}(\gr_\F^{(-\infty,-s]}M)$ (so there are no nonzero differentials out of the
    $(s,t)$-term).
\end{notation}

\begin{definition}[The $\E^\infty$-page]
    Assume that $\Cscr$ admits sequential limits and colimits (which implies that $\Cscr^\heart$
    admits countable unions and intersections).
    Let $\tilde{\B}^\infty_{s,t}=\colim_r\tilde{\B}^r_{s,t}$ and let
    $\tilde{\Z}^{\infty}_{s,t}=\lim_r\tilde{\Z}^r_{s,t}$, where the limits and
    colimits are computed in $\Cscr^\heart$.\footnote{Note that
    $\tilde{\Z}^\infty_{s,t}$ is necessarily a subobject of
    $\E^1_{s,t}\iso\pi_{s+t}\gr^{-s}$ since limits are left exact, but that
    the natural map $\tilde{\B}^\infty_{s,t}\rightarrow\E^1_{s,t}$ need not be
    injective in general. However, it will be injective as long as the
    $t$-structure on $\Cscr$ is compatible with sequential colimits.}
    The sequential limits and colimits can be obtained by computing them in
    $\Cscr$ and then applying $\pi_0$.
    The $\E^\infty$-page of the spectral
    sequence is the lattice of objects with
    $$\E^\infty_{s,t}=\mathrm{coker}\left(\tilde{\B}^\infty_{s,t}\rightarrow\tilde{\Z}^\infty_{s,t}\right).$$
\end{definition}

\begin{remark}
    If $\F^\star M$ is upper half-plane or left half-plane, then the
    $\E^\infty$-page can be accessed in a slightly different way. For example,
    if $\F^\star M$ is $(-N)$-connective in the Beilinson $t$-structure (so it
    is upper half-plane), then
    for $n\geq N+t+1$, no differential
    can hit $\E^n_{s,t}$ so one obtains a decreasing sequence
    $$\cdots\subseteq\E^{N+t+3}_{s,t}\subseteq\E^{N+t+2}_{s,t}\subseteq\E^{N+t+1}_{s,t}$$
    of objects of $\Cscr^\heart$.  Since $\F^\star M$ is upper half-plane or left half-plane, for any pair
    $(s,t)$, the boundary sequence $\tilde{\B}^r_{s,t}$ stabilizes, which
    implies that the two definitions of the $\E^\infty$-page agree.
    It follows that $\E^\infty_{s,t}$ is the intersection of these objects.
\end{remark}

Recall that a decreasing system $\G^\star N$ of objects of $\Cscr$ (and in
particular $\Cscr^\heart$) is pro-zero if for each $n\in\bZ$ there is an
$m\geq n$ such that $\G^mN\rightarrow\G^nN$ is nullhomotopic (zero if in
$\Cscr^\heart$).

\begin{lemma}\label{lem:prozero}
    Let $\Cscr$ be a stable $\infty$-category with sequential limits. 
    If $\G^\star N$ is a pro-zero filtered object of $\Cscr$, then it is a
    complete filtration, i.e., $\lim_s\G^s N\we 0$.
\end{lemma}

\begin{proof}
    Let $M\in\Cscr$. Then, $\F^\star X=\bMap_\Cscr(M,\F^\star N)$ is a pro-zero
    filtered spectrum and $\lim_s\F^s X=\bMap_\Cscr(M,\lim_s\F^s N)$. Since $M$
    was arbitrary, it follows that it is enough to show that $\lim_s\F^s X\we
    0$ if $\F^\star X$ is a pro-zero filtered spectrum. However, for each
    $n\in\bZ$, $\pi_n\F^\star X$ is a pro-zero filtered abelian group, so that
    ${\lim_s}^1\pi_n\F^s X=0=\lim_s\pi_n\F^s X$ by the Mittag--Leffler
    condition. By the Milnor sequence
    $$0\rightarrow{\lim_s}^1\pi_{n+1}\F^s
    X\rightarrow\pi_n\lim_s\F^sX\rightarrow\lim_s\pi_n\F^sX\rightarrow 0,$$
    which follows from the fact that the $t$-structure on spectra is compatible
    with countable products, it follows that $\pi_n\lim_s\F^sX\we 0$ for all
    $n$. In other words, $\lim_s\F^s X\we 0$.
\end{proof}

If $\Cscr$ is a stable $\infty$-category with sequential limits, derived limits
of objects of $\Cscr^\heart$ are always computed in $\Cscr$. This notion
depends on the ambient $\infty$-category $\Cscr$ and not just on the abelian
category $\Cscr^\heart$.

\begin{theorem}[Convergence theorem]\label{thm:convergence}
    Suppose that $\Cscr$ admits sequential limits and colimits and
    that the fixed $t$-structure on $\Cscr$ is compatible with sequential colimits.
    Let $\F^\star M$ be a filtration on an object $M$ of
    $\Cscr$.
    \begin{enumerate}
        \item[{\em (i)}] If $\F^\star M$ is an exhaustive filtration of $M$, then for each $n$ the induced filtration $\F^\star\pi_nM$ is
            exhaustive.
        \item[{\em (ii)}] If for some integer $n$ the pro-object $\{\pi_n\F^sM\}_s$ is
            pro-zero, then the filtration $\F^\star\pi_nM$ is pro-zero and
            derived complete.
        \item[{\em (iii)}] There are natural inclusions
            $$\gr^s\pi_nM\subseteq\E^\infty_{-s,n+s}$$
            for each $s$ and $n$. The inclusions are equalities if $\F^\star M$ is
            column-bounded, row-bounded, first quadrant, or third quadrant.
    \end{enumerate}
\end{theorem}

\begin{proof}
    The fact that $\F^\star\pi_nM$ is exhaustive follows from the compatibility
    of the $t$-structure on $\Cscr$ with sequential colimits and the fact that
    $\F^\star M$ is exhaustive, so $\colim_s\F^sM\we M$. This proves (i).

    For (ii), since $\pi_n\F^\star M\rightarrow\F^\star\pi_nM$ is a surjection of
    pro-objects of $\Cscr^\heart$ and since $\pi_n\F^\star M$ is pro-zero, it
    follows that $\F^\star\pi_nM$ is pro-zero and in particular derived
    complete by Lemma~\ref{lem:prozero}.

    Now, consider the commutative diagram
    $$\xymatrix{
        &&&0\ar[d]\\
        0\ar[r]&\ker(\pi_n\F^{s+1}M\rightarrow\pi_nM)\ar[r]\ar[d]&\pi_n\F^{s+1}M\ar[r]\ar[d]&\F^{s+1}\pi_nM\ar[r]\ar[d]&0\\
        0\ar[r]&\ker(\pi_n\F^sM\rightarrow\pi_nM)\ar[r]\ar[d]&\pi_n\F^sM\ar[r]\ar[d]&\F^s\pi_nM\ar[r]\ar[d]&0\\
        0\ar[r]&B\ar[r]\ar[d]&Z\ar[d]\ar[r]&\frac{\F^s\pi_nM}{\F^{s+1}\pi_nM}\ar[d]\ar[r]&0\\
        &0&0&0
    }$$
    of exact sequences, where $B$ and $Z$ are defined to be the vertical cokernels of
    the upper left square.
    This implies that $B$ and $Z$ are subobjects of $\pi_n\gr^sM$ and in fact they
    are subobjects of $\tilde{\Z}^\infty_{-s,s+n}$. Indeed,
    $Z\subseteq\tilde{\Z}^\infty_{-s,s+n}$ is the subobject of permanent cycles
    that actually lift to $\pi_n\F^sM$.
    
    The kernel $\ker(\pi_n\F^sM\rightarrow\pi_nM)$
    is the union of the kernels of $\pi_n\F^sM\rightarrow\pi_n\F^{s-N}M$ since
    the $t$-structure is compatible with sequential colimits. The commutative
    diagram
    $$\xymatrix{
        \pi_n\F^{s+1}M\ar[r]\ar[d]&\pi_n\F^{s+1}M\ar[d]\\
        \pi_n\F^sM\ar[r]\ar[d]&\pi_n\F^{s-N}M\ar[d]\\
        \pi_n\gr^sM\ar[r]&\pi_n\gr^{[s-N,s+1)}M,
    }$$ in which the columns are exact, implies that $\ker(\pi_n\F^sM\rightarrow\pi_nM)$ maps to
    $\ker(\pi_n\gr^sM\rightarrow\pi_n\gr^{[s-N,s+1)}M)$, which is
    $\tilde{\B}^n_{-s,s+n}$.
    It follows there is a commutative diagram of
    exact sequences
    $$\xymatrix{
        &0\ar[d]&0\ar[d]\\
        0\ar[r]&B\ar[r]\ar[d]&Z\ar[d]\\
        0\ar[r]&\tilde{\B}^\infty_{-s,s+n}\ar[r]&\tilde{\Z}^{\infty}_{-s,s+n}.
    }$$
    This diagram presents $B$ as a pullback.
    To prove this, write $B^N=\ker(\pi_n\F^sM\rightarrow\pi_n\F^{s-N}M)$. Consider the
    intersection $\B^N_{-s,s+n}\cap Z\subseteq\pi_n\gr^sM$ and let $C$ be the
    inverse image of the intersection in $\pi_n\F^sM$. By construction, the map
    $C\hookrightarrow\pi_n\F^sM\rightarrow\pi_n\F^{s-N}M\rightarrow\gr^{[s-N,s+1)}M$
    maps to zero, so $C$ maps to
    $\im(\pi_n\F^{s+1}M\rightarrow\pi_n\F^{s-N}M)$ in $\pi_n\F^{s-N}M$. Let $\tilde{C}$ be the
    pullback in
    $$\xymatrix{
        \tilde{C}\ar[r]\ar[d]&C\ar[d]\\
        \pi_n\F^{s+1}M\ar[r]&\im(\pi_n\F^{s+1}M\rightarrow\pi_n\F^{s-N}M).
    }$$
    Since the bottom map is an epimorphism, so is $\tilde{C}\rightarrow C$.
    Let $f$ be the composition $\tilde{C}\rightarrow
    C\hookrightarrow\pi_n\F^sM$ and
    let $g$ be the composition
    $\tilde{C}\rightarrow\pi_n\F^{s+1}M\rightarrow\pi_n\F^sM$.
    Consider the map $\tilde{C}\xrightarrow{g-f}\pi_n\F^sM$. By construction,
    $g-f$ maps $\tilde{C}$ into $B^N$. The kernel of $\tilde{C}\rightarrow C$
    is isomorphic to $\ker(\pi_n\F^{s+1}M\rightarrow\pi_n\F^{s-N}M)$, which
    maps to zero in $\pi_n\gr^sM$ (using any of the maps $f$, $g$, or $g-f$). Thus, there is a commutative square
    $$\xymatrix{
        \tilde{C}\ar[r]\ar[d]^{g-f}&C\ar[d]\\
        B^N\ar[r]&\B^N_{-s,s+n}\cap Z.
    }$$
    The top and right arrows are surjective and hence the bottom arrow is
    surjective as well. Since the $t$-structure is compatible with sequential
    colimits, $B\iso\cup\B^N\rightarrow\cup(\B^N_{-s,s+n}\cap
    Z)\iso(\cup\B^n_{-s,s+n})\cap Z$ is a surjection.

    It follows that the induced map
    $Z/B\rightarrow\tilde{\Z}^\infty_{-s,s+n}/\tilde{\B}^\infty_{-s,s+n}\iso\E^\infty_{-s,s+n}$
    is injective. It is an isomorphism if $Z\rightarrow\tilde{\Z}^{\infty}_{-s,s+n}$ is
    surjective. This is the case if $\F^\star M$ is column-bounded,
    row-bounded, first quadrant, or third quadrant since in those cases
    $\pi_n\F^sM\rightarrow\pi_n\F^{[s,s+N+1)}M$ is an isomorphism for large $N$.
    This proves (iii).
\end{proof}

\section{Indexing conventions}\label{sec:reindexing}

It should be understood that the choice of $\E^1_{s,t}=\pi_{s+t}\gr^{-s}M$ is
entirely a convention. One could make many other choices. This one is
the Serre grading convention.

\begin{figure}[b]
    \centering
    \makebox[\textwidth]{
        \begin{tabular}{||l l l l l||}
            \hline
            $\pi_*\F^\star
            M$&$\E^1_{s,t}=\pi_{s+t}\gr^{-s}M\Rightarrow\pi_{s+t}M$&$(-r,r-1)$&$\F^s{\pi_nM}=\im(\pi_n\F^sM\rightarrow\pi_nM)$&$\gr^s\pi_nM\iso\E^\infty_{-s,s+n}$\\
            \hline
            $\pi_*\F_\star
            M$&$\E^1_{s,t}=\pi_{s+t}\gr_sM\Rightarrow\pi_{s+t}M$&$(-r,r-1)$&$\F_s\pi_nM=\im(\pi_n\F_sM\rightarrow\pi_nM)$&$\gr_s\pi_nM\iso\E^\infty_{s,-s+n}$\\
            \hline
            $\H^*\F^\star
            M$&$\E_1^{s,t}=\H^{s+t}\gr^sM\Rightarrow\H^{s+t}M$&$(r,-r+1)$&$\F^s\H^nM=\im(\H^n\F^sM\rightarrow\H^nM)$&$\gr^s\H^nM\iso\E_\infty^{s,-s+n}$\\
            \hline
            $\H^*\F_\star
            M$&$\E_1^{s,t}=\H^{s+t}\gr_{-s}M\Rightarrow\H^{s+t}M$&$(r,-r+1)$&$\F_s\H^nM=\im(\H^n\F_sM\rightarrow\H^nM)$&$\gr_s\H^nM\iso\E_\infty^{-s,s+n}$\\
            \hline
        \end{tabular}
    }
    \caption{Spectral sequence conventions for decreasing filtrations $\F^\star$ and increasing
    filtrations $\F_\star$. We use $\Rightarrow$ even when no convergence is
    implied.}
    \label{fig:ssconventions}
\end{figure}

Sometimes it is useful to reindex a spectral sequence to start at the second
page instead of the first, in which case we write the resulting spectral sequence as $'\E^2$. This is a purely psychological move, but it leads to
some useful comparisons as one can see in
Section~\ref{sec:mysterious}. The matrices to go back and forth between the
$\E^1$ and $\E^2$-conventions are
$$\begin{pmatrix}0&-1\\1&2\end{pmatrix}\colon\text{from $\E^2$ to
    $\E^1$}\quad\text{and}\quad\begin{pmatrix}2&1\\-1&0\end{pmatrix}\colon\text{from $\E^1$ to $\E^2$}.$$

        \begin{figure}[b!]
    \centering
    \makebox[\textwidth]{
        \begin{tabular}{||l l l||}
            \hline
            $\pi_*\F^\star
            M$&$'\E^2_{s,t}=\E^1_{-t,s+2t}=\pi_{s+t}\gr^{t}M\Rightarrow\pi_{s+t}M$&$\gr^s\pi_nM\iso\E^\infty_{-s+n,s}$\\
            \hline
            $\pi_*\F_\star
            M$&$'\E^2_{s,t}=\E^1_{-t,s+2t}=\pi_{s+t}\gr_{-t}M\Rightarrow\pi_{s+t}M$&$\gr_s\pi_nM\iso\E^\infty_{s+n,-s}$\\
            \hline
            $\H^*\F^\star
            M$&$'\E_2^{s,t}=\E_1^{-t,s+2t}=\H^{s+t}\gr^{-t}M\Rightarrow\H^{s+t}M$&$\gr^s\H^nM\iso\E_\infty^{s+n,-s}$\\
            \hline
            $\H^*\F_\star
            M$&$'\E_2^{s,t}=\E_1^{-t,s+2t}=\H^{s+t}\gr_{t}M\Rightarrow\H^{s+t}M$&$\gr_s\H^nM\iso\E_\infty^{-s+n,s}$\\
            \hline
        \end{tabular}
    }
    \caption{Reindexing to $'\E^2$. The differentials have the same bidegree as
    in the $\E^1$-spectral sequences from Figure~\ref{fig:cover_ssconventions}. The
    filtrations on the abutment are also defined in the same way as in the $\E^1$ spectral
    sequences.}
    \label{fig:E2ssconventions}
\end{figure}

There is also a common reindexing in homotopy theory called the Adams grading
convention.
The beautiful charts one sees of the Adams spectral sequence are all given in
this convention. The matrices to go back and forth between the Serre and Adams
grading conventions are
$$\begin{pmatrix}0&-1\\1&1\end{pmatrix}\colon\text{from Adams to
    Serre}\quad\text{and}\quad\begin{pmatrix}1&1\\-1&0\end{pmatrix}\colon\text{from Serre to Adams}.$$
Note that Adams spectral sequences are not spectral sequences if one defines a
spectral sequence to have differentials of a certain bidegree.

The main advantage of Adams indexing is that all contributions to $\pi_n$ of
the abuttment are in the $n$th column and all differentials leaving or entering
this column relate to the $(n+1)$st or $(n-1)$st column. This makes displaying
spectral sequences more compact in some situations.

        \begin{figure}[b!]
    \centering
    \makebox[\textwidth]{
        \begin{tabular}{||l l l l l||}
            \hline
            $\pi_*\F^\star
            M$&$\E^1_{s,t}=\pi_{s}\gr^{t}M\Rightarrow\pi_{s}M$&$(-1,r)$&$\F^s{\pi_nM}=\im(\pi_n\F^sM\rightarrow\pi_nM)$&$\gr^s\pi_nM\iso\E^\infty_{n,s}$\\
            \hline
            $\pi_*\F_\star
            M$&$\E^1_{s,t}=\pi_{s}\gr_{-t}M\Rightarrow\pi_{s}M$&$(-1,r)$&$\F_s\pi_nM=\im(\pi_n\F_sM\rightarrow\pi_nM)$&$\gr_s\pi_nM\iso\E^\infty_{n,-s}$\\
            \hline
            $\H^*\F^\star
            M$&$\E_1^{s,t}=\H^{s}\gr^{-t}M\Rightarrow\H^{s}M$&$(1,-r)$&$\F^s\H^nM=\im(\H^n\F^sM\rightarrow\H^nM)$&$\gr^s\H^nM\iso\E_\infty^{n,-s}$\\
            \hline
            $\H^*\F_\star
            M$&$\E_1^{s,t}=\H^{s}\gr_{t}M\Rightarrow\H^{s}M$&$(1,-r)$&$\F_s\H^nM=\im(\H^n\F_sM\rightarrow\H^nM)$&$\gr_s\H^nM\iso\E_\infty^{n,s}$\\
            \hline
        \end{tabular}
    }
    \caption{Adams indexing spectral sequence conventions.}
    \label{fig:adamsssconventions}
\end{figure}

\section{Multiplicativity}\label{sec:multiplicativity}

Compare the material in this section to related results in Hedenlund's thesis~\cite{hedenlund-thesis}.

\begin{definition}
    The $\E^r$-page construction defines a functor
    $$\E^r\colon\F\Cscr\rightarrow\Gr(\Ch^\bullet(\Cscr^\heart))$$ from the $\infty$-category
    of filtrations in $\Cscr$ to the $1$-category of graded objects in cochain
    complexes in $\Cscr^\heart$. The weight $n$ piece of the
    $\E^r$-page is the cochain complex is the weight $n$ piece of the $\E^1$-page of
    $\Dec^{(r-1)}(\F)$, which is $$\E^1_{-\bullet,n}(\Dec^{(r-1)}(\F)).$$ Under the regrading, this
    corresponds to
    $$\E^r_{(r-1)n-r\cs,-(r-2)n+(r-1)\cs},$$
    which is the line of slope $\frac{1-r}{r}=-1+\frac{1}{r}$ through the point
    of height $n-\frac{n}{r}$ on the $t$-axis. This cochain complex is indexed
    so that the cohomological degree $0$ term is $\E^r_{(r-1)n,-(r-2)n}$ and the
    cohomological degree $1$ term is $\E^r_{(r-1)n-r,-(r-2)n+(r-1)}$.
    Note that the cohomological degree $n$ term in weight $n$ is $\E^r_{-n,n}$.
    See Figure~\ref{fig:weight1generalss} for an illustration of the weight $1$
    term on the first three pages of a general spectral sequence.
\end{definition}

The choice of indexing for the $\E^r$-page is not arbitrary but is enforced by the
d\'ecalage approach to spectral sequences; see Theorem~\ref{thm:decalage}.
We note the following lemma for the readers convenience.

\begin{lemma}\label{lem:lookup}
    The term $\E^r_{s,t}$ is the cohomological degree $(r-2)s+(r-1)t$ term in
    weight $(r-1)s+rt$. 
\end{lemma}

\begin{sseqdata}[name = generalss, homological Serre grading, classes = {draw = none } ]
    \foreach \x in {-7,...,3} {
        \foreach \y in {-2,...,6} {
            \class["\bullet"](\x,\y)
        }
    }
    \foreach \x in {-4,...,0} {
        \foreach \y in {0,...,4} {
            \d1(\x,\y)
            \replacetarget["\bullet"]
        }
    }

    \foreach \y in {0,...,4} {
        \d1(1,\y)
        \replacetarget["\bullet"]
        \replacesource["\bullet"]
    }

    \foreach \x in {-4,...,0} {
        \foreach \y in {0,...,4} {
            \d2(\x,\y)
            \replacetarget["\bullet"]
        }
    }

    \foreach \x in {-4,...,0} {
        \d2(\x,-1)
        \replacetarget["\bullet"]
        \replacesource["\bullet"]
    }

    \foreach \x in {1,2} {
        \foreach \y in {-1,...,4} {
            \d2(\x,\y)
            \replacetarget["\bullet"]
            \replacesource["\bullet"]
        }
    }
    \foreach \x in {-4,...,3} {
        \foreach \y in {0,...,4} {
            \d3(\x,\y)
        }
    }
    \foreach \x in {-1,0,1,2,3} {
        \d3(\x,-2)
        \d3(\x,-1)
    }
    
    \d3(-2,-2)
    \d3(-2,-1)
    \d3(-3,-1)
\end{sseqdata}

\begin{figure}[h]
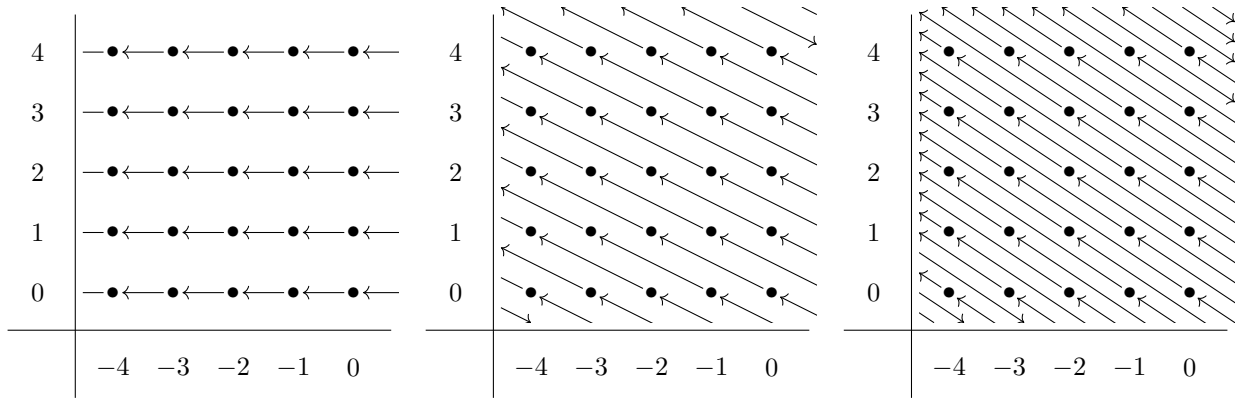

    \centering
    \printpage[name = generalss, page = 1,xscale=0.8,yscale=0.8,x range = {-4}{0},y range = {0}{4}]
    \quad
    \printpage[name = generalss, page = 2,xscale=0.8,yscale=0.8,x range = {-4}{0},y range = {0}{4}]
    \quad
    \printpage[name = generalss, page = 3,xscale=0.8,yscale=0.8,x range = {-4}{0},y range = {0}{4}]
    \caption{The $\E_1$, $\E_2$, and $\E_3$-pages of a general spectral sequence.}
    \label{fig:generalss}
\end{figure}

\begin{sseqdata}[name = weight1page1, homological Serre grading, classes = {draw = none } ]
    \class["\bullet"](-3,1)
    \class["\bullet"](-2,1)
    \class["\bullet"](-1,1)
    \class["\circ"](0,1)
    \class["\bullet"](1,1)
    \class["\bullet"](2,1)
    \class["\bullet"](3,1)
    \foreach \x in {-2,...,3} {
        \d1(\x,1)
    }
\end{sseqdata}

\begin{sseqdata}[name = weight1page2, homological Serre grading, classes = {draw = none } ]
    \class["\circ"](1,0)
    \class["\bullet"](-1,1)
    \class["\bullet"](-3,2)
    \class["\bullet"](-5,3)
    \class["\bullet"](3,-1)
    \class["\bullet"](5,-2)
    \class["\bullet"](7,-3)

    \d2(7,-3)
    \d2(5,-2)
    \d2(3,-1)
    \d2(1,0)
    \d2(-1,1)
    \d2(-3,2)
\end{sseqdata}

\begin{sseqdata}[name = weight1page3, homological Serre grading, classes = {draw = none } ]
    \class["\circ"](2,-1)
    \class["\bullet"](-1,1)
    \class["\bullet"](-4,3)
    \class["\bullet"](-7,5)
    \class["\bullet"](5,-3)
    \class["\bullet"](8,-5)
    \class["\bullet"](11,-7)

    \d3(11,-7)
    \d3(8,-5)
    \d3(5,-3)
    \d3(2,-1)
    \d3(-1,1)
    \d3(-4,3)
\end{sseqdata}

\begin{figure}[h!]
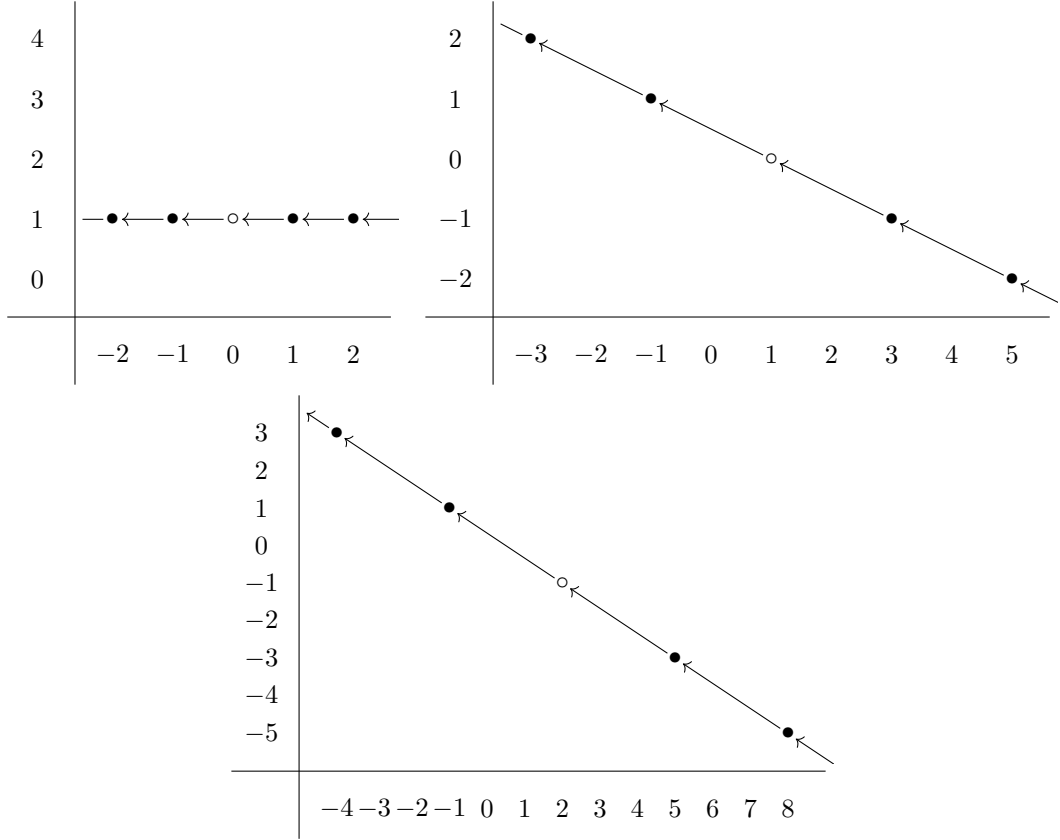

    \centering
    \printpage[name = weight1page1, page = 1,xscale=0.8,yscale=0.8,x range = {-2}{2},y range = {0}{4}]
    \quad
    \printpage[name = weight1page2, page = 2,xscale=0.8,yscale=0.8,x range = {-3}{5},y range = {-2}{2}]
    \quad
    \printpage[name = weight1page3, page = 3,xscale=0.5,yscale=0.5,x range = {-4}{8},y range = {-5}{3}]
    \caption{The weight $1$ parts of the $\E_1$, $\E_2$, and $\E_3$-pages of a general spectral
    sequence. The $\circ$ class is degree $0$ of the cochain complex.}
    \label{fig:weight1generalss}
\end{figure}

The reader is
referred to~\cite[Chap.~2]{ha} for all conventions about symmetric monoidal structures
and the related theory of $\infty$-operads.

For a symmetric monoidal structure $\Cscr^\otimes$ on an
$\infty$-category $\Cscr$, the tensor product of two objects $X$ and $Y$ is
written $X\otimes_\Cscr Y$. In the case where $\Cscr=\Perf(k)^\otimes$ or $\Cscr=\D(k)^\otimes$
for an $\bE_\infty$-ring $k$, this notation will be simplified to $X\otimes_kY$.

Let $\Cscr$ be a stable $\infty$-category.
If $\Cscr$ admits both a symmetric monoidal structure $\Cscr^\otimes$ such that
the tensor product is exact in each variable and a
$t$-structure $(\Cscr_{\geq 0},\Cscr_{\leq 0})$, the
$t$-structure is said to be compatible with the symmetric monoidal structure if
$\Cscr_{\geq 0}$ contains the unit object $\1_\Cscr$ and is closed under tensor
products. In this case, the prestable $\infty$-category $\Cscr_{\geq 0}$ inherits a symmetric monoidal
structure $\Cscr_{\geq 0}^\otimes$ and the inclusion $\Cscr_{\geq 0}\subseteq\Cscr$ is symmetric
monoidal.

Suppose that $\Cscr$ is a stable $\infty$-category with a symmetric monoidal structure
$\Cscr^\otimes$ which is exact in each variable and suppose also that $\Cscr$ admits sequential
colimits and that the tensor product preserves sequential colimits in each variable.
Using that $\bZ^\op$ naturally admits a symmetric monoidal structure given by
addition of integers, there is an induced Day convolution symmetric
monoidal structure $\F\Cscr^\otimes$ on the $\infty$-category
$\Fun(\bZ^\op,\Cscr)$ of decreasing filtrations in $\Cscr$. This is constructed
in general in the $\infty$-categorical context by Glasman~\cite{glasman}; see
also~\cite[Ex.~2.2.6.17]{ha}.
To describe the tensor product of any given pair of decreasing filtrations
$\F^\star$ and $\G^\star$, one considers the left Kan extension
$$\xymatrix{
    \bZ^\op\times\bZ^\op\ar[rrr]^{(s,s')\mapsto\F^s\otimes_\Cscr\F^{s'}}\ar[d]^+&&&\Cscr\\
    \bZ^\op.\ar@{.>}[urrr]
}$$
The left Kan extension exists by our assumption that $\Cscr$ admits sequential colimits and hence
all countable colimits since $\Cscr$ is stable.

The constant filtrations $\Cscr\subseteq\F\Cscr$ form a $\otimes$-ideal in
$\F\Cscr$ since if $\F^\star$ is the constant filtration on $M$, then
$\F^\star\otimes_{\F\Cscr}\G^\star$ is the constant filtration on
$M\otimes_\Cscr\G^{-\infty}$, as one sees from the left Kan extension
definition above. It follows that if $\Cscr$ admits sequential limits, then
$\widehat{\F}\Cscr$ admits a symmetric monoidal structure
$\widehat{\F}\Cscr^\otimes$. The tensor product here is the completed tensor
product $\widehat{\otimes}$ obtained by computing the tensor product in $\F\Cscr$
and then completing.

\begin{warning}
    The category $\Xi^\op$ which classifies cochain complexes is also symmetric
    monoidal. However, the induced Day convolution symmetric monoidal structure
    on $\widehat{\F}\Cscr\we\Fun(\Xi^\op,\Cscr)$ does not agree with the
    symmetric monoidal structure on $\widehat{\F}\Cscr$ induced by completing
    the Day convolution symmetric monoidal structure on $\Fun(\bZ^\op,\Cscr)$.
    Indeed, as pointed out by Thomas Nikolaus, this does not even work in
    the $1$-categorical situation of ordinary cochain complexes of abelian
    groups. The unit object of $\Fun(\Xi^\op,\Ab)$ is the cochain complex
    $\cdots\rightarrow0\rightarrow\bZ\xrightarrow{\id}\bZ\rightarrow
    0\rightarrow \cdots$, where the non-trivial entries are in cohomological degrees $0$ and $1$.
\end{warning}

Left Kan extension along $\bZ^\delta\rightarrow\bZ^\op$ induces a symmetric
monoidal functor $\ins^\star\colon\Gr\Cscr^\otimes\rightarrow\F\Cscr^\otimes$.

\begin{proposition}\label{prop:grtensor}
    The associated graded functor $\F\Cscr\rightarrow\Gr\Cscr$ naturally upgrades
    to a symmetric monoidal functor
    $\F\Cscr^\otimes\rightarrow\Gr\Cscr^\otimes$ and the induced composition
    $\Gr\Cscr^\otimes\rightarrow\F\Cscr^\otimes\rightarrow\Gr\Cscr^\otimes$ is
    equivalent to the identity as a symmetric monoidal functor.
\end{proposition}

\begin{proof}
    This is not a formal result and is proved in~\cite[Prop.~3.2.1]{lurie-rotation} when
    $\Cscr\we\Sp$. That proof works in general.
\end{proof}

The proposition gives one of the main ways to understand the tensor product on
filtered objects. Indeed, the tensor product on $\Gr\Cscr^\otimes$ is straightforward.
If $X(\star)$ and $Y(\star)$ are graded objects in $\Cscr$, then
$$(X\otimes_{\Gr\Cscr}
Y)(s)\we\bigoplus_{i+j=s} X(i)\otimes_\Cscr Y(j).$$
For example, this is used in the next lemma to understand the compatibility of
the Day convolution symmetric monoidal structure on $\F\Cscr$ with Beilinson
$t$-structures.

\begin{lemma}
    Suppose that $\Cscr$ is a stable $\infty$-category admitting sequential
    limits and colimits and equipped with a symmetric monoidal structure $\Cscr^\otimes$ whose
    tensor product is exact and preserves sequential colimits in each variable. If
    $\Cscr$ is additionally equipped with a $t$-structure $(\Cscr_{\geq
    0},\Cscr_{\leq 0})$ compatible with $\Cscr^{\otimes}$, then the induced
    Beilinson $t$-structure on $\F\Cscr$ is compatible with the Day convolution
    symmetric monoidal structure on $\F\Cscr^\otimes$.
\end{lemma}

\begin{proof}
    The unit object of $\F\Cscr^\otimes$ is $\ins^0\1_\Cscr$ which has
    $\gr^0\ins^0\1_\Cscr\we\1_\Cscr$ and $\gr^s\ins^0\1_\Cscr\we 0$ for
    $s\neq 0$. This object is connective in the Beilinson $t$-structure. The
    tensor product formula
    $\gr^s(\F^\star\otimes_{\F\Cscr}\G^\star)\we\oplus_{i+j=s}\gr^i_\F\otimes_\Cscr\gr^j_\G$
    shows that if $\F^\star$ and $\G^\star$ are connective in the Beilinson
    $t$-structure, then so is $\F^\star\otimes_{\F\Cscr}\G^\star$.
\end{proof}

If $\Cscr$ is a stable $\infty$-category equipped with a symmetric monoidal
structure $\Cscr^\otimes$ whose tensor product is exact in each variable and a compatible $t$-structure $(\Cscr_{\geq 0},\Cscr_{\leq
0})$, then $\Cscr^\heart$ inherits a symmetric monoidal structure $\Cscr^{\heart,\otimes}$. Indeed, by
definition, the $\infty$-category $\Cscr_{\geq 0}$ of connective objects
inherits a symmetric monoidal structure $\Cscr_{\geq 0}^\otimes$ by compatibility. The unit object of
$\Cscr^{\heart,\otimes}$ is
$\pi_0\1_\Cscr$ and the tensor product is formed by forming the tensor product
in $\Cscr_{\geq 0}$ and applying $\pi_0$.

\begin{example}[Koszul symmetric monoidal structure]
    Let $\Cscr$ be a symmetric monoidal stable $\infty$-category with sequential colimits whose
    tensor product is exact and preserves sequential colimits in
    each variable with a compatible $t$-structure.
    Then, $\Gr\Cscr$ admits a symmetric monoidal structure via Day convolution and a compatible
    $t$-structure where a graded object $X(\star)$ is connective if $X(i)\in\Cscr_{\geq i}$ for all
    $i\in\bZ$. The heart $(\Gr\Cscr)^\heart$ is equivalent to the abelian category
    $\Gr(\Cscr^\heart)$ of graded
    objects in $\Cscr^\heart$ and the induced symmetric monoidal structure agrees with the usual
    monoidal structure on graded objects, with a symmetric braiding given by the Koszul sign rule.
    We will call this the Koszul symmetric monoidal structure on $\Gr(\Cscr^\heart)$.
\end{example}

\begin{lemma}\label{lem:homotopylaxmonoidal}
    Suppose that $\Cscr$ is a stable $\infty$-category with sequential colimits equipped with a symmetric
    monoidal structure $\Cscr^\otimes$ whose tensor product is exact and preserves sequential colimits in each variable and a compatible $t$-structure $(\Cscr_{\geq
    0},\Cscr_{\leq 0})$. Then, the functors $$\tau_{\geq\star}\colon\Cscr\rightarrow\F\Cscr$$ and
    $$\pi_*\colon\Cscr\rightarrow\Gr(\Cscr^\heart)$$
    admit lax symmetric monoidal structures, where $\F\Cscr$ is equipped with the Day convolution
    symmetric monoidal structure and $\Gr(\Cscr^\heart)$ is equipped with
    the Koszul symmetric monoidal structure.
\end{lemma}

\begin{proof}
    Consider the $t$-structure on $\F\Cscr$ with connective objects 
    those filtrations $\F^\star$ such that $\F^s$ is in $\Cscr_{\geq s}$ for all
    $s\in\bZ$. This is not a Beilinson $t$-structure, but it is still
    compatible with the Day convolution symmetric monoidal structure on
    $\Cscr$. It is called the positive $t$-structure by Raksit
    in~\cite{raksit}. We will call it the Postnikov $t$-structure and denote its connective objects
    by $(\F\Cscr)_{\geq 0}^\P$. In particular, since the inclusion $\F\Cscr_{\geq
    0}^\P\subseteq\F\Cscr$ is symmetric monoidal, the right adjoint $\tau_{\geq
    0}^\P$ is lax symmetric monoidal. The heart of the positive $t$-structure is
    $\Gr(\Cscr^\heart)$ with the Koszul sign rule. Now, $\tau_{\geq\star}$ is the composition of
    the lax symmetric monoidal functors
    $$\Cscr\xrightarrow{\text{constant}}\F\Cscr\xrightarrow{\tau_{\geq
    0}^\P}\F\Cscr_{\geq 0}^\P$$ and
    $\pi_*$ is the composition of the lax symmetric monoidal functors
    $$\Cscr\xrightarrow{\text{constant}}\F\Cscr\xrightarrow{\tau_{\geq
    0}^\P}\F\Cscr_{\geq 0}^\P\xrightarrow{\pi_0^\P}\Gr(\Cscr^\heart).$$
    This completes the proof.
\end{proof}

\begin{construction}[The $\E^r$-monoidal structure]
    Let $\Cscr^\otimes$ be a stable symmetric monoidal $\infty$-category which admits sequential
    limits and colimits, whose tensor product is exact and preserves sequential colimits in each
    variable, and which is equipped with a compatible $t$-structure. We view the
    $\E^r$-page functor as a functor from $\F\Cscr$ to $\Gr(\Ch^\bullet(\Cscr^\heart))$. 
    We want to view the latter category as being symmetric monoidal. We do this by observing that
    it is the heart of $\F(\F\Cscr)$ with respect to the $t$-structure on bifiltrations
    $\F^{\star,\star}$ where the connective objects are those where $\F^{i,\star}$ is
    $i$-connective in the Beilinson $t$-structure. This is $t$-structure is ``half-Postnikov, half-Beilinson.''
    Call this the $\E^1$-symmetric monoidal structure
    $\Gr(\Ch^\bullet(\Cscr^\heart))^{\otimes_{\E^1}}$. The $\E^r$-monoidal structure is obtained by
    transport of structure via the reindexing functor given by
    $$\begin{pmatrix}0&-1\\1&2\end{pmatrix}^{-r+1}.$$
\end{construction}

\begin{theorem}\label{thm:multiplicativity}
    Let $\Cscr^\otimes$ be a stable symmetric monoidal $\infty$-category which admits sequential
    limits and colimits, whose tensor product is exact and preserves sequential colimits in each
    variable, and which is equipped with a compatible $t$-structure. 
    For each $r\geq 1$, the $\E^r$-page functor upgrades to a lax symmetric monoidal functor
    $\F\Cscr^\otimes\rightarrow\Gr(\Ch^\bullet(\Cscr^\heart))^{\otimes_{\E^r}}$.
\end{theorem}

\begin{proof}
    For $r=1$, the result follows by applying Lemma~\ref{lem:homotopylaxmonoidal} with respect to
    the Beilinson $t$-structure. Now, we claim that the functor
    $\F^\star\mapsto\tau_{\geq\star}^\B(\F)$ from $\F\Cscr$ to $\F\F\Cscr$ is lax symmetric
    monoidal. This is another application of Lemma~\ref{lem:homotopylaxmonoidal}. Thus,
    $\Dec^{(r-1)}$ is lax symmetric monoidal by composition. It follows that
    $\F^\star\mapsto\E^1(\Dec^{(r-1)}(\F))$ is lax symmetric monoidal with target the
    $\E^1$-symmetric monoidal structure on $\Gr(\Ch^\bullet(\Cscr^\heart))$.
    The theorem follows by reindexing.
\end{proof}

\begin{corollary}\label{cor:multiplicativity}
    Let $\Cscr^\otimes$ be a stable symmetric monoidal $\infty$-category which admits sequential
    limits and colimits, whose tensor product is exact and preserves sequential colimits in each
    variable, and which is equipped with a compatible $t$-structure. Suppose that $\Oscr$ is an
    $\infty$-operad in $\Cscr_{\geq 0}^\otimes$.
    For each $r\geq 1$ the $\E^r$-page functor upgrades
    to a functor
    $$\E^r\colon\Alg_\Oscr(\F\Cscr^\otimes)\rightarrow\Alg_\Oscr(\Gr(\Ch^\bullet(\Cscr^\heart))^{\otimes_{\E^r}}).$$
\end{corollary}

\begin{example}[Rings]
    If $\F^\star
    R$ is a filtered $\bE_1$-ring spectrum, then each page $\E^r$ of the
    spectral sequence for $\F^\star R$ is an associative algebra object in
    $\Gr(\Ch^\bullet(\Cscr^\heart))$. If $\F^\star R$ is an $\bE_\infty$-ring
    spectrum, then each page $\E^r$ of the spectral sequence for $\F^\star R$
    is a commutative algebra object in $\Gr(\Ch^\bullet(\Cscr^\heart))$. These
    cases follow from Theorem~\ref{thm:multiplicativity} applied to the $\bE_1$
    and $\bE_\infty$ operads.

    To make this more explicit, if $\F^\star R$ is an $\bE_\infty$-ring, then
    $\E^r_{-r\cs,(r-1)\cs}$
    is a cdga in $\Cscr^\heart$ and the graded collection
    $$\{\E^r_{(r-1)t-r\cs,-(r-2)t+(r-1)\cs}\}_{t\in\bZ}$$ forms a graded-commutative
    ring in graded cochain complexes, meaning that there are morphisms
    $\E^r_{s,t}\otimes\E^r_{s',t'}\rightarrow\E^r_{s+s',t+t'}$ satisfying the
    bigraded Koszul sign rule $xy=(-1)^{(s+t)(s'+t')}yx$ and satisfying the following Leibniz
    rule for the horizontal differentials:
    $\d^1(xy)=\d^1(x)y+(-1)^{s+t}x\d^1(y)$ if $x\in\E^r_{s,t}$.
    These sign rules become the ordinary sign rules on the total complex
    associated to this graded cochain complex. These are justified by the fact
    $\E^r_{s,t}$ has weight $(n-1)s+nt$ and cohomological degree
    $(n-2)s+(n-1)t$ by Lemma~\ref{lem:lookup}, so the total degree $s+t$ has
    the same parity as $(n-1)s+nt+(n-2)s+(n-1)t=(2n-3)s+(2n-1)t$.
\end{example}

\section{The Atiyah--Hirzebruch spectral sequence}\label{sec:mysterious}

Let $X$ be a CW complex, which we view as an anima $X$ with an exhaustive increasing
filtration $\F_\star X$ where $\F_sX=\emptyset$ for $s<0$ and each
$\F_sX/\F_{s-1}$ is equivalent to a bouquet of
$s$-spheres, called the $s$-cells of $X$, for $s\geq 0$.
Taking singular cochains $\R\Gamma(X,\bZ)$ we obtain an object of $\D(\bZ)$ which we equip with
the decreasing filtration
$$\F^s\R\Gamma(X,\bZ)=\mathrm{fib}\left(\R\Gamma(X,\bZ)\rightarrow\R\Gamma(X_{s-1},\bZ)\right).$$
This filtration is complete and exhaustive with graded pieces given by
$\gr^s\R\Gamma(X,\bZ)\we\prod_{\text{$s$-cells}}\bZ[-s]$.
The associated (cohomological) spectral sequences is
$$\E_1^{s,t}=\H^{s+t}\gr^s\R\Gamma(X,\bZ)\iso\H^{s+t}\prod_{\text{$s$-cells}}\bZ[-s]\iso\begin{cases}
    \prod_{\text{$s$-cells}}\bZ&\text{if $t=0$ and}\\
    0&\text{otherwise.}
\end{cases}$$
The spectral sequence is thus concentrated on the ray from $(0,0)$ to
$(\infty,0)$, giving a cochain complex
$$0\rightarrow\prod_{\text{$0$-cells}}\bZ\rightarrow\prod_{\text{$1$-cells}}\bZ\rightarrow\prod_{\text{$2$-cells}}\bZ\rightarrow\cdots.$$
The differentials in this cochain
complex are precisely the differentials in the cochain complex computing the CW
cohomology of $X$. In particular, $\pi_0^\B(\F^\star\R\Gamma(X,\bZ))$ is the
CW cohomology cochain complex.
Thus, $\E^2_{s,0}\iso\E^\infty_{s,0}\iso\H^s_\CW(X,\bZ)$ and the
spectral sequence collapses at the $\E^2$-page. One finds that
$$\Dec(\F)^s\R\Gamma(X,\bZ)\we\begin{cases}
    0&\text{if $s>0$ and}\\
    \R\Gamma(X,\bZ)&\text{if $s\leq 0$.}
\end{cases}$$

\begin{remark}[Multiplicative structures and CW cohomology]
    It is well-known that the $\bE_\infty$-algebra $\R\Gamma(X,\bZ)$ cannot generally be computed
    by a cdga. This implies
    that for such $X$ no CW filtration $\F^\star\R\Gamma(X,\bZ)$ can be given the
    structure of an $\bE_\infty$-algebra object in $\FD(\bZ)$. If it could be,
    then $\pi_0^\B$ would admit the structure of a commutative algebra object by the results of
    Section~\ref{sec:multiplicativity}; this would be a cdga computing $\R\Gamma(X,\bZ)$.
\end{remark}

Suppose that $X$ is a CW complex and fix a
spectrum $M$. We will write $\H^i(M)=\pi_{-i}M$ below for convenience.
There are two classical ways to construct a spectral sequence for computing $\pi_*\R\Gamma(X,M)$,
either by filtering $M$ or by filtering $X$.

First, consider the cellular filtration given as in the previous section by
$$\F^s\R\Gamma(X,M)=\mathrm{fib}\left(\R\Gamma(X,M)\rightarrow\R\Gamma(X_{s-1},M)\right).$$
The associated graded pieces are
$$\gr^s_\F\R\Gamma(X,M)\we(\R\Gamma(X_{s-1},M)/\R\Gamma(X_{s},M))[-1]\we\prod_{\text{$s$-cells}}M[-s].$$
The associated spectral sequence has
$$\E_1^{s,t}(\F)=\H^{s+t}(\prod_{\text{$s$-cells}}M[-s])\iso\prod_{\text{$s$-cells}}\H^t(M)\Rightarrow\H^{s+t}(X,M).$$
By d\'evissage one can reduce to the case when $M\we\bZ$ to compute the
differential on the $\E_1$-page as the CW cohomology differential.
It follows that $$\E_2^{s,t}(\F)=\H^s(X,\H^t(M)).$$

Second, consider the Whitehead tower $$\G^\star M=\tau_{\geq\star}M$$
of $M$ and take cochains to obtain a new filtration
$\G^\star\R\Gamma(X,M)=\R\Gamma(X,\tau_{\geq\star}M)$ with graded pieces
$\gr^s_\G\R\Gamma(X,M)\we\R\Gamma(X,\gr^sM)\we\R\Gamma(X,(\H^{-s}M)[s])$.
The $\E_1$-page of the associated spectral sequence is
$$\E_1^{s,t}(\G)=\H^{s+t}(X,(\H^{-s}M)[s])\iso\H^{2s+t}(X,\H^{-s}M)\Rightarrow\H^{s+t}\R\Gamma(X,M).$$
Via the standard reindexing of an $\E_1$-spectral sequence to an
$'\E_2$-spectral sequence discussed in Section~\ref{sec:reindexing}, one obtains
an $'\E_2$-spectral sequence
$$'\E_2^{s,t}(\G)=\H^s(X,\H^t(M))\Rightarrow\H^{s+t}(X,M).$$
It is natural to ask for the relationship between these two spectral sequences.

\begin{proposition}
    There is an equivalence
    $\Dec(\F)^\star\R\Gamma(X,M)\we\G^\star\R\Gamma(X,M)$ of filtered spectra, functorial in $M$
    and the CW complex structure on $X$.
\end{proposition}

\begin{proof}
    Consider the bifiltration $\H^{a,b}=\fib\left(\R\Gamma(X,\tau_{\geq
    a}M)\rightarrow\R\Gamma(X_{s-1},\tau_{\geq a}M)\right)$. If $a$ is fixed, then
    $\tfrac{\H^{a,s}}{\H^{a,s+1}}\we\prod_{\text{$s$-cells}}(\tau_{\geq a} M)[-s]$, which is $(a-s)$-connective. In
    particular, the natural map $\H^{a,\star}\rightarrow\F^\star$ factors through $\tau_{\geq
    a}^\B(\F)^\star$. It is in fact an equivalence. Indeed, both $\H^{a,\star}$ and $\tau_{\geq
    a}^\B(\F)^\star$ are complete filtrations (the latter because $\F^\star$ is, see
    Lemma~\ref{lem:truncationscomplete}), so it suffices to check this on associated graded pieces. However,
    $$\frac{\H^{a,s}}{\H^{a,s+1}}\we\prod_{\text{$s$-cells}}(\tau_{\geq a}M)[-s]\we\tau_{\geq
    a-s}\left(\prod_{\text{$s$-cells}}M[-s]\right)\we\tau_{\geq a-s}\gr^s_\F\we\gr^s\tau_{\geq a}^\B(\F),$$
    where the second equivalence follows from the compatibility of the $t$-structure on $\Sp$ with
    arbitrary products. It follows that $\H^{a,\star}\we\tau_{\geq a}^\B(\F)$ and hence that
    $|\H^{a,\star}|\we\Dec(\F)^a$. However, $|\H^{a,\star}|\we\G^a$. This completes the proof.
\end{proof}

The following consequence was proved by Maunder in~\cite[Thm.~3.3]{maunder} with different methods.

\begin{corollary}
    There is an isomorphism of spectral sequences $\E^r_{s,t}(\F)\iso{'\E^r_{s,t}(\G)}$ starting at
    the second page.
\end{corollary}

One often says that $\E^r(\F)$ and $'\E^r(\G)$ agree from the $\E^2$-page on.

\bibliographystyle{amsplain}
\bibliography{ss}
\addcontentsline{toc}{section}{References}

\medskip
\noindent
\textsc{Department of Mathematics, Northwestern University}\\
\textsc{Max-Planck-Institut für Mathematik Bonn}\\
{\ttfamily antieau@northwestern.edu}

\end{document}

%% file: preamble.tex
\usepackage{amsmath}
\usepackage{amscd}
\usepackage{amsbsy}
\usepackage{amssymb}
\usepackage{verbatim}
\usepackage{eucal}
\usepackage{microtype}
\usepackage{mathrsfs}
\usepackage{amsthm}
\usepackage{stmaryrd}
\usepackage{wasysym}
\usepackage{xy}
\input xy
\xyoption{all}
\usepackage{tikz}
\usetikzlibrary{matrix,arrows}
\usepackage{tikz-cd}
\usepackage{dsfont}
\usepackage[T1]{fontenc}
\usepackage{spectralsequences}
\usepackage[pdfstartview=FitH,
            colorlinks,
            linkcolor=reference,
            citecolor=citation,
            urlcolor=e-mail,
            bookmarks=false,
            ]{hyperref}

\usepackage{fancyhdr}
\newcommand{\stackspace}{2.5}
\newcommand{\stack}[2][1cm]{\;\tikz[baseline, yshift=.65ex]%
    {\foreach \k [evaluate=\k as \r using (.5*#2+.5-\k)*\stackspace] in {1,...,#2}{%
    \ifodd\k{\draw[->](0,\r pt)--(#1,\r pt);}%
    \else{\draw[<-](0,\r pt)--(#1,\r pt);}\fi
    }}\;}

\usepackage[bbgreekl]{mathbbol}
\usepackage{amsfonts}
\DeclareSymbolFontAlphabet{\mathbb}{AMSb} % to ensure \mathbb does not change
\DeclareSymbolFontAlphabet{\mathbbl}{bbold}

\usepackage{color}
\definecolor{todo}{rgb}{1,0,0}
\definecolor{conditional}{rgb}{0,1,0}
\definecolor{e-mail}{rgb}{0,.40,.80}
\definecolor{reference}{rgb}{.20,.60,.22}
\definecolor{mrnumber}{rgb}{.80,.40,0}
\definecolor{citation}{rgb}{0,.40,.80}

\renewcommand{\bf}{\bfseries}

\let\oldmarginpar\marginpar
\renewcommand\marginpar[1]{\-\oldmarginpar[\raggedleft\footnotesize #1]%
{\raggedright\footnotesize #1}}

\newcommand{\cs}{\bullet}
\newcommand{\fs}{\star}

\newcommand{\Ascr}{\mathcal{A}}

\newcommand{\Cscr}{\mathcal{C}}
\newcommand{\Dscr}{\mathcal{D}}

\newcommand{\Oscr}{\mathcal{O}}

\newcommand{\Sscr}{\mathcal{S}}

\newcommand{\B}{\mathrm{B}}

\renewcommand{\d}{\mathrm{d}}
\newcommand{\D}{\mathrm{D}}
\newcommand{\E}{\mathrm{E}}
\newcommand{\F}{\mathrm{F}}
\newcommand{\G}{\mathrm{G}}
\renewcommand{\H}{\mathrm{H}}

\renewcommand{\L}{\mathrm{L}}

\renewcommand{\P}{\mathrm{P}}
\newcommand{\R}{\mathrm{R}}
\renewcommand{\r}{\mathrm{r}}

\newcommand{\Z}{\mathrm{Z}}

\renewcommand{\1}{\mathbf{1}}
\newcommand{\bA}{\mathbf{A}}

\newcommand{\bE}{\mathbf{E}}

\newcommand{\bG}{\mathbf{G}}

\newcommand{\bS}{\mathbf{S}}

\newcommand{\bZ}{\mathbf{Z}}

\newcommand{\Sp}{\Sscr\mathrm{p}}

\newcommand{\SW}{\mathrm{SW}}

\newcommand{\Ab}{\Ascr\mathrm{b}}

\newcommand{\op}{\mathrm{op}}
\newcommand{\cofib}{\mathrm{cofib}}
\newcommand{\fib}{\mathrm{fib}}

\newcommand{\Mod}{\mathrm{Mod}}
\newcommand{\LMod}{\mathrm{LMod}}

\newcommand{\qc}{\mathrm{qc}}

\newcommand{\Perf}{\mathrm{Perf}}
\newcommand{\Ind}{\mathrm{Ind}}

\newcommand{\Alg}{\mathrm{Alg}}

\newcommand{\FD}{\mathrm{FD}}
\newcommand{\FDhat}{\widehat{\FD}}

\newcommand{\Gr}{\mathrm{Gr}}
\newcommand{\gr}{\mathrm{gr}}
\newcommand{\ev}{\mathrm{ev}}
\newcommand{\ins}{\mathrm{ins}}

\newcommand{\MU}{\mathrm{MU}}

\newcommand{\CW}{\mathrm{CW}}
\newcommand{\Dec}{\mathrm{Dec}}

\newcommand{\heart}{\heartsuit}

\newcommand{\id}{\mathrm{id}}

\newcommand{\im}{\mathrm{im}}
\renewcommand{\geq}{\geqslant}
\renewcommand{\leq}{\leqslant}
\newcommand{\Ho}{\mathrm{Ho}}

\newcommand{\TP}{\mathrm{TP}}

\newcommand{\Ch}{\mathrm{Ch}}

\newcommand{\Map}{\mathrm{Map}}
\newcommand{\bMap}{\mathbf{Map}}

\newcommand{\Hom}{\mathrm{Hom}}
\newcommand{\Fun}{\mathrm{Fun}}

\newcommand{\Gm}{\bG_{m}}

\DeclareMathOperator*{\colim}{colim}

\DeclareMathOperator*{\Tot}{Tot}

\DeclareMathOperator{\Spec}{Spec}

\newcommand{\we}{\simeq}
\newcommand{\iso}{\cong}

\theoremstyle{plain}
\newtheorem{theorem}{Theorem}[section]
\newtheorem*{theorem*}{Theorem}
\newtheorem{lemma}[theorem]{Lemma}

\newtheorem{proposition}[theorem]{Proposition}

\newtheorem{corollary}[theorem]{Corollary}
\newtheorem*{corollary*}{Corollary}

\theoremstyle{plain}

\theoremstyle{definition}

\newtheoremstyle{named}{}{}{\itshape}{}{\bfseries}{.}{.5em}{#1 \thmnote{#3}}
\theoremstyle{named}

\theoremstyle{definition}

\newtheorem{definition}[theorem]{Definition}
\newtheorem{warning}[theorem]{Warning}

\newtheorem{notation}[theorem]{Notation}

\newtheorem{fact}[theorem]{Fact}
\newtheorem{example}[theorem]{Example}
\newtheorem*{example*}{Example}

\newtheorem*{question*}{Question}
\newtheorem{construction}[theorem]{Construction}

\newtheorem{remark}[theorem]{Remark}